\documentclass[12pt]{amsart}
\usepackage{style,macros}

\title{Reconstruction of graphs and surfaces from derived categories}
\date{\today}

\author{Oscar Harr}
\address{Matematiska institutionen\\Stockholms universitet\\106 91 Stockholm\\Sweden}
\email{oscar.harr@math.su.se}


\begin{document}
\begin{abstract}
  A pair of topological spaces $X$ and $Y$ is a \emph{Fourier--Mukai (FM) pair} if there is an equivalence of $\infty$-categories of derived sheaves of abelian groups \[\Shv X[\mathrm D(\ZZ )]\simeq\Shv Y[\mathrm D(\ZZ )].\] An FM pair $(X,Y)$ is \emph{non-trivial} if $X$ and $Y$ are not homeomorphic. We consider the question of whether non-trivial FM pairs $(X,Y)$ exist for which $X$ and $Y$ are both ``nice'' spaces, which in this article means admitting a triangulation. We show that the answer is no if $X$ is second-countable and either (i) one-dimensional, (ii) a surface, or (iii) the 3-sphere. In order to prove this, we define two invariants of dualizable categories. The first invariant is an extension of \emph{Tate--Hochschild cohomology} (aka singular Hochschild cohomology) to a class of dualizable categories satisfying a countability assumption. For a homology manifold $X$, we relate the Tate--Hochschild cohomology of $\Shv X[\mathrm D(\mathbb Z)]$ to the homology at infinity of $X$ (in the sense of proper homotopy theory). The second invariant takes the form of a subcategory $\reg{\mathcal D}\subseteq\mathcal D$ called the \emph{Gorenstein locus}, defined for every dualizable stable category $\mathcal D$. For a triangularizable space $X$ and a field $k$, we show that the Gorenstein locus of $\Shv X[\mathrm D(k)]$ is equal to the full subcategory of sheaves supported on the $k$-homology manifold locus of $X$.
\end{abstract}
\maketitle
\tableofcontents
\thispagestyle{empty}
\newpage
\section{Introduction}
\vspace*{0.5cm}%
\begin{flushleft}
  \hspace*{.05\textwidth}%
  \begin{minipage}{0.7\textwidth}
    \begin{footnotesize}
      ``Hallo!'' said Piglet, ``what are \textit{you} doing?'' ($\ldots$) \\
      ``Tracking something,'' said Winnie-the-Pooh very mysteriously.\\
      ``Tracking what?''\\
      ``That's just what I ask myself. I ask myself, What?''\\
      \hphantom{space} --- \textit{Winnie-the-Pooh}, A.A.~Milne
    \end{footnotesize}
  \end{minipage}
\end{flushleft}
\medskip The starting point of this article is the following question:
\begin{ques}\label{question:information-about-topology}
  What can we learn about a topological space by looking only at its derived category of sheaves?
\end{ques}
As a prototypical example of a result with this theme, Efimov has shown
\begin{nonothm}[{Efimov~\cite[Thm~6.1]{Efimov2025}}]
  Let $X$ be a locally compact Hausdorff space and let $\mathcal C$ be a dualizable stable $\infty$-category, and let $L\colon\Prdual\to\mathcal T$ be a finitary localizing invariant.
  There is an equivalence
  \begin{equation}
    \label{eq:localizing-invariant-of-sheaves}
    L(\Shv X[\mathcal C])
    \simeq\cptGlob (X,L(\mathcal C)_X),
  \end{equation}
  natural in proper maps and open immersions.
\end{nonothm}
Here $L(\mathcal C)_X$ denotes the constant sheaf on $X$ with value $L(\mathcal C)$. We can read the theorem as saying that $\Shv X[\mathcal C]$ remembers something about locally constant sheaves on $X$. Also in this spirit, applying the compact objects functor ${-}^{\aleph_0}\colon\Prdual\to\smallcats$ to $\Shv X[\mathcal C]$ recovers the full subcategory spanned by locally constant $\mathcal C$-valued sheaves on $X$ that have compact stalks and compact support \cites[][Thm~2.6]{Harr2025}[][Prop~6.15]{Efimov2025}.
\subsection{Statement of results}
In this article, we introduce two invariants of dualizable categories and use them to retrieve information about topological spaces from their derived categories.

The first invariant that we consider is \tdef{Tate--Hochschild cohomology}, which we denote $\mdef{\contsgHH{-}}$. Previously, several authors have studied the Tate--Hochschild cohomology of various types of associative algebras~\cite{Eu_Schedler2009,Bergh_Jorgensen2013,Wang2021,Keller2018,Rivera_Wang2019}. We define Tate--Hochschild cohomology for dualizable stable $\infty$-categories subject to a condition called \tdef{second-countability} (\cref{defn:second-countable}). (The derived category of a locally compact Hausdorff space $X$ is second-countable if and only if $X$ is second-countable, i.e.~admits a countable basis of opens.) Our main result about Tate--Hochschild cohomology is
\begin{mainthm}\label{mainthm:tate-hochschild-cohomology}
  Let $k$ be a commutative ring. For a second-countable $d$-dimensional $k$-homology manifold $X$, there is an equivalence
  \begin{equation}
    \label{eq:intro-sing-hochschild-of-topological-space}
    \contsgHH{\Shv X[\mathrm D(k)]}\simeq
      \mathrm C^\infty_\bullet (X;k)\lbrack -2d+1\rbrack,
  \end{equation}
  where $\mathrm C^\infty_\bullet (X;k)$ is the complex of singular chains at infinity (recalled in~\cref{subsec:topology}).
\end{mainthm}
Outside topology, the derived category $\mathrm D(A)$ of an associative $k$-algebra $A$ is second-countable if and only if $A$ admits a possibly infinite resolution by finite-rank projectives as a module over the enveloping algebra $A\otimes_{k} A^{\mathrm{op}}$. For these algebras, our definition of Tate--Hochschild cohomology agrees with the established one. Note that in algebraic geometry Tate--Hochschild cohomology measures non-smoothness, whereas~\cref{mainthm:tate-hochschild-cohomology} shows that in topology it is instead related to non-properness.

\medskip

Our second main result is perpendicular to the results on locally constant sheaves mentioned at the start of this article, insofar as it revolves around the \emph{failure} of the Verdier dualizing sheaf to be locally constant near singularities of $X$. For every dualizable stable $\infty$-category $\mathcal D$, we define a full subcategory $\mdef{\reg{\mathcal D}}\subseteq\mathcal D$ called the \tdef{Gorenstein locus} of $\mathcal D$. Intuitively, the Gorenstein locus consists of objects of $\mathcal D$ that behave as if they are ``supported on smooth points'' when probed by the Serre functor on $\mathcal D$.
\begin{mainthm}\label{mainthm:gorenstein-locus}
  Let $k$ be a field, let $X$ be a locally compact triangularizable space, and put
  \[
    \mdef{\hreg X} = \lbrace x\in X\mid\dualizingshf X\text{ is invertible in a neighborhood of x}\rbrace.
  \]
  The Gorenstein locus $\reg{\Shv X[\mathrm D(k)]}\subseteq\Shv X[\mathrm D(k)]$ is equal to the full subcategory of sheaves supported on $\hreg X$.
\end{mainthm}
Note that $\hreg X\subseteq X$ is by construction a (typically non-compact) $k$-homology manifold, and in fact it is the largest open subspace of $X$ with this property. The theorem shows that, at least for finite dimensional locally compact triangularizable spaces,~\cref{question:information-about-topology} can be split into studying derived categories of homology manifolds and the extension problems that arise when gluing homology manifold strata back together.

The subspace $\hreg X$ also appears as the top stratum in the canonical homology manifold stratifications constructed by Goresky--MacPherson~\cite[\S~4.2]{Goresky_Macpherson1983} to prove the topological invariance of intersection cohomology. It would be interesting to know whether the remaining strata (for some choice of perversity) can also be characterized in terms of the category $\Shv X[\mathrm D(k)]$. Note that this would imply a reasonable definition of intersection cohomology for dualizable categories.

\medskip

As in the abstract, say that a pair of (non-homeomorphic) topological spaces $(X,Y)$ is a \tdef{(non-trivial) FM pair} if there is an equivalence of $\infty$-categories 
\[\Shv X[\mathrm D(\ZZ )]\simeq\Shv Y[\mathrm D(\ZZ )].\]
Non-sober topological spaces supply many ``trivial'' non-trivial FM pairs. For instance, $(X,X\times Y)$ is an FM pair for any topological space $X$ and anti-discrete\footnote{Recall that a topological space $Y$ is \tdef{anti-discrete} its only open subsets are $\varnothing$ and $Y$.} space $Y$. A more interesting source of examples comes from the theory of Verdier duality for posets~\cite{Curry2018,Aoki2023}, as we explain in~\cref{exam:verdier-duality-for-posets}. From our point of view, all of these examples are point-set topology pathologies. Among nice topological spaces (which for concreteness and convenience we take to mean triangularizable spaces), we do not know of a single non-trivial FM pair. By applying~\cref{mainthm:tate-hochschild-cohomology,mainthm:gorenstein-locus}, we prove that such non-trivial FM pairs do not exist for some low-dimensional spaces:
\begin{mainthm}\label{mainthm:reconstruction}
  Let $X$ be a locally compact triangularizable space which is second-countable and either 
  \begin{enumerate}[label=(\roman*)]
  \item one-dimensional;
  \item a finite type surface\footnote{Recall that a 2-dimensional manifold, aka a surface, is called \tdef{finite type} if it is the complement of a union of boundary components inside a compact surface.}; or
  \item a 3-sphere.
  \end{enumerate}
  There are no FM pairs of the form $(X,Y)$ with $Y$ locally compact triangularizable.
\end{mainthm}
One can also interpret this theorem as saying that, among triangularizable spaces, spaces of the form (i), (ii), and (iii) can be reconstructed up to homeomorphism from their derived categories. Both these statements and the terminology of FM pairs are inspired by algebraic geometry, especially by a well-known theorem of Bondal and Orlov~\cite{Bondal_Orlov2001}. We will comment on the relationship between the proof of Bondal and Orlov's theorem and our~\cref{mainthm:gorenstein-locus} in~\cref{sec:gorenstein-locus} below. Beyond Bondal and Orlov's theorem, there is a rich theory on the algebro-geometric analog of~\cref{question:information-about-topology}, see~\cite{Orlov2018} for an overview. This is one of several interrelated subjects known as non-commutative geometry.

The main obstacle to extending non-commutative geometry to derived categories of topological spaces is the scarcity of compact objects in the latter; in sharp contrast to the situation in algebraic geometry, derived categories of topological spaces are almost never compactly generated, see~\cite[Prop~3.1]{Harr2025} for a precise statement. Because of this, many techniques from the algebro-geometric story cannot be applied to our topological question, or to other dualizable categories such as those that arise from almost mathematics and formal geometry. For example, none of the three non-commutative notions of dimension considered in~\cite{Elagin_Lunts2021} make sense for non-compactly generated categories. However,~\cref{mainthm:gorenstein-locus} will allow us to determine the dimension of a triangularizable space from its derived category in a somewhat ad hoc way.
\begin{rmk}
  Aoki has shown that the functor
  \begin{equation}
    \label{eq:aoki-tannaka-functor}
    \Shv{-}[\mathrm D(k)]\colon \lbrace\text{sober topological spaces}\rbrace^{\mathrm{op}}\too\CAlg (\Prlst )
  \end{equation}
  is fully faithful for any field $k$~\cite[Thm~B]{Aoki2024}. In particular, if $X$ and $Y$ are sober topological spaces such that $\Shv X[\mathrm D(\ZZ )]\simeq\Shv Y[\mathrm D(\ZZ )]$ \emph{as symmetric monoidal $\infty$-categories} (rather than just as $\infty$-categories), then $X$ and $Y$ must be homeomorphic. It also follows that the functor
  \[
    \Shv{-}[\mathrm D(\ZZ )]\colon \lbrace\text{sober topological spaces}\rbrace^{\mathrm{op}}\too\largecats
  \]
  is conservative, since it can be written as the composite of~\eqref{eq:aoki-tannaka-functor} and the conservative forgetful functor $\CAlg (\Prlst )\to\largecats$. Just as in algebraic geometry,\footnote{where the analog of Aoki's theorem is Balmer's reconstruction theorem~\cite{Balmer2002}} the subtlety of FM partners arises from the possible existence of ``exotic'' equivalences $\Shv X[\mathrm D(\ZZ )]\simeq \Shv Y[\mathrm D(\ZZ )]$ that are not combinations of functors induced by maps of underlying spaces and the operation of tensoring with invertible local systems.  
\end{rmk}
\subsection{Acknowledgements}
Part (i) in \Cref{mainthm:reconstruction} was initially stated in my PhD thesis at the University of Copenhagen, in which an incorrect proof is given. I thank Shaul Barkan for many energetic (but ultimately unsuccessful) conversations about constructing interesting Fourier--Mukai pairs in topology. Thanks also to Achim Krause for asking me about reconstruction of two-dimensional spaces, which led me to~\cref{mainthm:tate-hochschild-cohomology} and part (ii) of~\cref{mainthm:reconstruction}. I learned about \v{S}ne\u{i}der's theorem from a heuristic search\footnote{Specifically, the prompt was ``Let $X$ be a locally compact Hausdorff space. Assume that both $X$ and $(X\times X)\setminus\Delta$ are $\sigma$-compact, where $\Delta\subseteq X\times X$ denotes the diagonal. Does it follow that $X$ is second-countable?''} using Qwen 3.7 Max. We learned about Tate--Hochschild cohomology from a similar search\footnote{``Suppose $A$ is a ring which is an $\omega_1$-compact (aka $\aleph_1$-compact) as an object of the derived $\infty$-category $\mathrm D(A\otimes A^{op})$. Then $A$ can be written as a sequential colimit $A\simeq\varinjlim_n P_n$ of some sequence of compact $A\otimes A^{\mathrm{op}}$-modules (aka ``perfect complexes'') $P_n$, and furthermore $\lbrace P_n\rbrace$ is essentially unique as an Ind-object in $\mathrm D(A\otimes A^{\mathrm{op}})$. What is the Ind-object $\lbrace \cofib ( P_n) \to A\rbrace$? Is it related to any invariants defined in the literature?''} using ChatGPT 5.6. Aside from these inputs, LLMs were only used to proofread and test some arguments.

I was supported by Dan Petersen's Wallenberg Scholar fellowship and the Danish National Research Foundation through the Copenhagen Centre for Geometry and Topology (DNRF151).

\section{Notation and terminology}
\begin{enumerate}[wide,labelwidth=!,labelindent=10pt,itemsep=10pt,label=\Roman*.]
\item We use the term `category' to refer to an $\infty$-category. If we wish to stress that a category is an old-fashioned category (i.e.~has homotopy discrete mapping spaces), we will call it a 1-category.
\item For convenience, we assume the existence of uncountable Grothendieck universes $\mathcal U\in\mathcal U'\in\mathcal U''$ whose elements we refer to as \tdef{small}, \tdef{large}, and \tdef{very large} respectively. Sets, topological spaces, spectra, modules, etc.~are small, whereas categories are (unless otherwise stated) large.
\item For an $E_\infty$-ring spectrum $R$ and an $R$-linear stable category $\mathcal C$, we let \[\mdef{\Hom{-}{-}[\mathcal C]}\colon\mathcal C^{\mathrm{op}}\times\mathcal C\to\modules[R]\] denote the $\modules[R]$-valued mapping object functor. For a topological space $X$, we let
  \[
    \mdef{\shHom{-}{-}[X]}\colon\Shv X[\modules[R]]^{\mathrm{op}}\otimes \Shv X[\modules[R]]\longrightarrow\Shv X[\modules[R]]
  \]
  denote the internal mapping object (aka Hom sheaf).
\item Let $X$ be a topological space and let $\mathcal C$ be a stable category. For a sheaf $F\in\Shv X[\mathcal C]$, we refer to the subset
  \[
    \mdef{\supp F} = \lbrace x\in X\mid F_x\not\simeq 0\rbrace
  \]
  as the \tdef{support} of $F$. (Many authors use the term ``support'' to refer to the closure of this set!)
\item For a locally compact Hausdorff space $X$ and an $E_\infty$-ring spectrum $R$, \[\mdef{\dualizingshf X}\in\Shv X[\modules[\basering]]\] denotes the Verdier dualizing sheaf. That is, $\dualizingshf X = p^!\basering$ where $p\colon X\to\pt$ is the projection to a point. Similarly, for $M\in\modules[\basering]$, we write \[\mdef{M_X}\in\Shv X[\modules[R]]\] for the constant sheaf with value $M$, i.e.~$M_X = p^*M$.
\item For a locally compact Hausdorff space $X$ and an open subset $U\subseteq X$, we let \[\mdef{h_U} = \basering\lbrack \yo (U)\rbrack\in\Shv X[\modules[R]]\] denote the $\basering$-linearization of the representable sheaf associated with $U$. Equivalently, $h_U = j_!\basering_U$, where $j\colon U\hookrightarrow X$ is the inclusion.
\item A \tdef{triangulation} of a topological space $X$ is a pair\[(\Sigma ,f\colon |\Sigma |\to X ),\] where $\Sigma$ is an abstract simplicial complex, $|{-}|$ denotes geometric realization, and $f$ is a homeomorphism. A topological space equipped with a distinguished triangulation (resp.~admitting a triangulation) is called \tdef{triangulated} (resp.~\tdef{triangularizable}). If $X$ is a triangulated space with triangulation $f$, we transfer combinatorial notions (simplices, subcomplexes, subdivisions, etc.) to $X$ in the obvious way.
\end{enumerate}

\section{Tate--Hochschild cohomology of large categories}
\begin{figure}[h]
  \centering
  \def\svgwidth{0.8\textwidth}
\begingroup%
  \makeatletter%
  \providecommand\color[2][]{%
    \errmessage{(Inkscape) Color is used for the text in Inkscape, but the package 'color.sty' is not loaded}%
    \renewcommand\color[2][]{}%
  }%
  \providecommand\transparent[1]{%
    \errmessage{(Inkscape) Transparency is used (non-zero) for the text in Inkscape, but the package 'transparent.sty' is not loaded}%
    \renewcommand\transparent[1]{}%
  }%
  \providecommand\rotatebox[2]{#2}%
  \newcommand*\fsize{\dimexpr\f@size pt\relax}%
  \newcommand*\lineheight[1]{\fontsize{\fsize}{#1\fsize}\selectfont}%
  \ifx\svgwidth\undefined%
    \setlength{\unitlength}{283.46456693bp}%
    \ifx\svgscale\undefined%
      \relax%
    \else%
      \setlength{\unitlength}{\unitlength * \real{\svgscale}}%
    \fi%
  \else%
    \setlength{\unitlength}{\svgwidth}%
  \fi%
  \global\let\svgwidth\undefined%
  \global\let\svgscale\undefined%
  \makeatother%
  \begin{picture}(1,0.34599998)%
    \lineheight{1}%
    \setlength\tabcolsep{0pt}%
    \put(0.75854157,0.08337186){\color[rgb]{0,0,0}\makebox(0,0)[lt]{\lineheight{1.25}\smash{\begin{tabular}[t]{l}$K_n$\end{tabular}}}}%
    \put(0.83527024,0.12835127){\color[rgb]{0,0,0}\makebox(0,0)[lt]{\lineheight{1.25}\smash{\begin{tabular}[t]{l}$X\setminus K_n$\end{tabular}}}}%
    \put(0,0){\includegraphics[width=\unitlength,page=1]{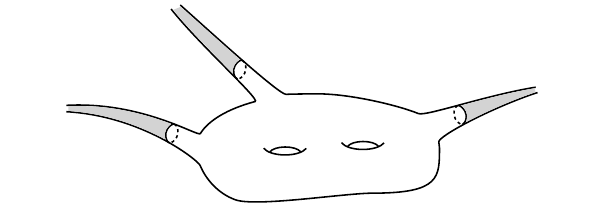}}%
  \end{picture}%
\endgroup%

  \caption{A surface with three ends, depicted as cusps, and with the three components of $X\setminus K_n$ (for $n\gg 0$) highlighted in gray}
  \label{fig:ends-genus-two-three-cusps}
\end{figure}
Let $X$ be an \tdef{$\sigma$-compact} Hausdorff space, which we recall means that it admits a \tdef{compact exhaustion}; that is, a filtration
\begin{equation}
  \label{eq:compact-exhaustion-classical}
  K_1\subseteq K_2\subseteq\cdots\subseteq X
\end{equation}
such that each $K_i$ is compact, $K_i$ is contained in the interior of $K_{i+1}$, and $X = \bigcup_1^\infty K_i$. Recall that an \tdef{end} of $X$ is a point in the pro-discrete space
\begin{equation}
  \label{eq:ends-pro-discrete}
  \limit_\NN\pi_0 (X\setminus K_i),
\end{equation}
where $\pi_0$ denotes components (not path components). Up to canonical homeomorphism, the space~\labelcref{eq:ends-pro-discrete} does not depend on the choice of the compact exhaustion~\labelcref{eq:compact-exhaustion-classical}. In particular, the cardinality
\[
  \#\lbrace\text{ends of }X\rbrace
\]
is a well-defined invariant of $X$.

The goal of this section is to categorify the theory of ends. Our main input is Clausen's theory of compactly-exhaustible objects in dualizable categories, which we recall in in~\cref{subsec:compactly-exhaustible} below. Using this theory, we construct an end-theoretic invariant for dualizable categories subject to a countability condition. Incidentally, this invariant generalizes the notion of Tate--Hochschild cohomology (aka singular Hochschild cohomology) for associative algebras in the sense of~\cite{Eu_Schedler2009,Bergh_Jorgensen2013,Wang2021}, and we can therefore view it as a kind of continuous extension of the latter to a class of large (i.e.~not compactly-generated) dualizable categories.
\subsection{Recollection on compactly exhaustible objects}\label{subsec:compactly-exhaustible}
Naively, one might hope to categorify the notion of a compact exhaustion such as~\labelcref{eq:compact-exhaustion-classical} by considering filtrations by compact objects. But derived categories of sheaves on topological spaces have very few compact objects; and for non-compact spaces (the only spaces for which there is anything to say about ends), there are no non-zero compact objects at all~\cites[][]{Neeman2001}[][Thm~2.6]{Harr2025}[][Prop~6.15]{Efimov2025}.\footnote{More precisely, any compact object in the derived category of a connected non-compact locally compact Hausdorff space is zero.} Instead, compact exhaustions in dualizable categories are defined in terms of a relative notion compactness:
\begin{defn}[{Clausen~\cites[][]{Clausen2023}[][\S~2.2]{Krause_Nikolaus_Putzstuck2024}}]\label{defn:compact-maps-and-exhaustions}
  Let $\mathcal D$ be a presentable stable category.
  \begin{enumerate}[label=(\roman*)]
  \item A morphism $x\to y$ in $\mathcal D$ is \tdef{compact} if for every filtered diagram $\lbrace z_j\rbrace_{j\in I}$ in $\mathcal D$ and every morphism $y\to\colimit_j z_j$, the composite
    \[
      x\longrightarrow y\longrightarrow\colimit_j z_j
    \]
    factors through the canonical map $z_i\to\colimit_j z_j$ for some $i\in I$.
  \item An object $x\in\mathcal C$ is \tdef{compactly exhaustible} if it can be written
    \begin{equation*}
      x\simeq\colimit\left( x_0\too x_1\too x_2\too\cdots\right),
    \end{equation*}
    where each transition map $x_i\to x_{i+1}$ is compact.
  \end{enumerate}
\end{defn}
\begin{rmk}
  In~\cite{Krause_Nikolaus_Putzstuck2024}, a morphism satisfying (i) above is called \emph{weakly compact}. There is a related concept of a morphism being \emph{strongly compact}. For dualizable stable categories, these notions coincide~\cite[Prop~2.4.1]{Krause_Nikolaus_Putzstuck2024}, so there is no reason for us to fuss over this distinction here.
\end{rmk}
The concept of a compactly exhaustible objects in a presentable stable category was used by Clausen to give the following \emph{internal} characterization of dualizable stable categories:\footnote{The usual definition, namely that of being a dualizable object $\mathcal D\in\Prlst$ with respect to the Lurie tensor product, is external in the sense that it makes no reference to objects or morphisms in $\mathcal D$.}
\begin{thm}[Clausen]\label{thm:internal-characterization-of-dualizability}
  Let $\mathcal D$ be a presentable stable category.
  \begin{enumerate}[label=(\roman*)]
  \item $\mathcal D$ is dualizable if and only if it is generated under colimits by compactly exhaustible objects.
  \item Suppose that $\mathcal D$ is a dualizable stable category. An object $x\in\mathcal D$ is compactly exhaustible if and only if it is $\aleph_1$-compact. If $x$ is compactly exhaustible, the ind-object underlying a sequential diagram $x_0\to x_1\to x_2\to\cdots$ as in~\cref{defn:compact-maps-and-exhaustions}(ii) is essentially unique.
  \end{enumerate}
\end{thm}
\begin{proof}
  This is Theorem 2.2.11 and Corollary 2.3.2 in~\cite{Krause_Nikolaus_Putzstuck2024}.
\end{proof}
\begin{rmk}\label{rmk:compactness-in-cpt-gen-categories}
  If $\mathcal D$ is a compactly-generated stable category, then a morphism $f\colon x\to y$ in $\mathcal D$ is compact if and only if $f$ factors through a compact object.
  It follows that an object $x\in\mathcal D$ is compactly exhaustible if and only if it is a sequential colimit of compact objects, which also follows directly from~\cref{thm:internal-characterization-of-dualizability} in this situation.  
\end{rmk}

\begin{prop}
  Let $X$ be a Hausdorff space. If $U$ and $V$ be open subsets of $X$ such that $U$ is contained in a compact subset of $V$, then the counit map
  \[
    h_U\longrightarrow h_V
  \]
  is compact.
\end{prop}
\begin{proof}
  This is shown in the proof of the (3) $\Longrightarrow$ (1) part of~\cite[Prop~2.2.20]{Krause_Nikolaus_Putzstuck2024}.
\end{proof}

\subsection{Ends of second-countable categories}\label{subsec:general-definitions}
Let $R$ be an $E_\infty$-ring spectrum. We let
\[
  \mdef{\Prlst[R]} = \operatorname{Mod}_{\mathrm{Mod}_R}(\Prlst )
\]
denote the (very large) symmetric monoidal category of \tdef{$R$-linear stable categories} with the $R$-linear Lurie tensor product $\otimes_R$. As a special case of a theorem of Arinkin, Gaitsgory, Kazhdan, Raskin, Rozenblyum, and Varshavsky, an $R$-linear stable category $\mathcal D$ is dualizable as an object of $\Prlst[R]$ if and only if the underlying stable category of $\mathcal D$ is dualizable~\cite[Thm~4.3.1]{Krause_Nikolaus_Putzstuck2024}. We refer to~\cite{Ramzi2024} for a comprehensive treatment of the theory of dualizable categories linear over a base such as $\modules[R]$.

\begin{rmk}[dg categories]
  In what follows, we will frequently restrict our attention to the situation in which $R$ is the Eilenberg--MacLane spectrum $Hk$ of an ordinary commutative $k$. In that case $\modules[Hk]$ is the unbounded derived $\infty$-category $\mathrm D(k)$, and $\Prlst[Hk]$ is folklorically equivalent to a category of $k$-linear dg categories.

  So as best to make a precise statement that is not too far from the existing literature, note that
  \[
    \Prlst[Hk]\simeq\colimit_\kappa\mathcal P\mathrm r_{Hk,\kappa ,\mathrm{st}}^L,
  \]
  where $\kappa$ ranges over small regular cardinals and $\mathcal P\mathrm r_{Hk,\kappa ,\mathrm{st}}^L\subset\Prlst[Hk]$ denotes the category of $\kappa$-compactly generated stable $Hk$-linear categories and $\kappa$-strongly continuous $Hk$-linear functors. We can therefore always assume that our categories and functors are in $\mathcal P\mathrm r_{Hk,\kappa ,\mathrm{st}}^L$ for a sufficiently large $\kappa$.

  Cohn has shown that there is an equivalence
  \begin{equation}
    \label{eq:cohn-equivalence}
    \mathrm{Cat}^k_{\mathrm{dg}}\lbrack W^{-1}\rbrack\xrightarrow{\quad\sim\quad}\mathcal P\mathrm r_{Hk,\aleph_0 ,\mathrm{st}}^L,
  \end{equation}
  where $\mathcal P\mathrm r_{Hk,\aleph_0 ,\mathrm{st}}^L$ is the category of $Hk$-linear compactly generated stable categories and strongly continuous functors, $\mathrm{Cat}^k_{\mathrm{dg}}$ is the 1-category of pretriangulated idempotent-complete $k$-linear dg categories, $W$ is the class of DK-equivalences, and $\mathrm{Cat}^k_{\mathrm{dg}}\lbrack W^{-1}\rbrack$ is the obvious Dwyer--Kan localization (which is no longer a 1-category)~\cite[Cor~5.5]{Cohn2016}.

  The arguments of~\cite{Cohn2016} can be extended to show an $\kappa$-compactly generated analog of~\labelcref{eq:cohn-equivalence} for any uncountable regular cardinal $\kappa >\aleph_0$, which takes the form
  \[
    \mathrm{Cat}^k_{\mathrm{dg},\kappa}\lbrack W^{-1}\rbrack\xrightarrow{\quad\sim\quad}\mathcal P\mathrm r_{Hk,\kappa ,\mathrm{st}}^L,
  \]
  where $\mathrm{Cat}^k_{\mathrm{dg},\kappa}$ is the 1-category of pretriangulated $k$-linear dg categories admitting all $\kappa$-small homotopy coproducts and $W$ is again the class of DK-equivalences. In particular, the homotopy theory of $\mathcal P\mathrm r_{Hk,\kappa ,\mathrm{st}}^L$ is the homotopy theory of $\mathrm{Cat}^k_{\mathrm{dg},\kappa}$ with the DK-model structure (as in~\cite[Thm~2.8]{Cohn2016}); they have the same homotopy categories, and (co)limits in $\mathcal P\mathrm r_{Hk,\kappa ,\mathrm{st}}^L$ are computed as homotopy (co)limits in $\mathrm{Cat}^k_{\mathrm{dg},\kappa}$.
\end{rmk}
As in the theory of ends in topology, our treatment of ends will require a countability assumption. As we shall see, however, our countability assumption corresponds in topology to a stricter assumption than $\sigma$-compactness.\footnote{Incidentally, it is exactly the assumption that Freudenthal imposed in his original treatment of ends~\cite{Freudenthal1931}.}
\begin{defn}\label{defn:second-countable}
  A dualizable $\basering$-linear stable category $\mathcal D$ is \tdef{second-countable} if the coevaluation
  \[
    \coev{\mathcal D}\colon\modules\longrightarrow\mathcal D\otimes_\basering\mathcal D^\vee
  \]
  is $\aleph_1$-strongly continuous. 
\end{defn}
The terminology will be justified in~\cref{thm:second-countability-recovers-notion-from-topology} below.
\begin{rmk}
  Since $\modules$ is generated by a single compact object $\basering$, the condition is equivalent to requiring that the object $\coev{\mathcal D}(\basering )\in\mathcal D\otimes_\basering\mathcal D^\vee$ is $\aleph_1$-compact.
\end{rmk}
\begin{rmk}
  By~\cite[Thm~C.6]{Efimov2025}, a dualizable stable category $\mathcal D$ is $\aleph_1$-compact as an object of $\Prdual$ if and only if both its evaluation $\ev{\mathcal D}$ and coevaluation $\coev{\mathcal D}$ are $\aleph_1$-strongly continuous. We only need the weaker condition of second-countability.
\end{rmk}
Suppose that $\mathcal D$ is second-countable. Since $\mathcal D\otimes_\basering\mathcal D^\vee$ is itself a dualizable $\basering$-linear stable $\infty$-category, it follows from~\cref{thm:internal-characterization-of-dualizability} that we can write
\[
  \coev{\mathcal D}(\basering )\simeq\colimit\left(
    \Phi_1\too\Phi_2\too\Phi_3\to\cdots
  \right),
\]
where $\Phi_n\to\Phi_{n+1}$ is a compact morphism in $\mathcal D\otimes_\basering\mathcal D^\vee$ for each $n$.
\begin{defn}\label{defn:homology-at-infty}
  Let $\mathcal D$ and $\coev{\mathcal D}$ be as above, and write $\coev{\mathcal D}^R$ for the right adjoint of $\coev{\mathcal D}$. 
  The \tdef{Tate--Hochschild cohomology} of $\mathcal D$, denoted $\mdef{\contsgHH{\mathcal D}}$, is defined to be the cofiber of the map
  \begin{equation}
    \label{eq:Hochschild-homology-of-large-cat}
    \colimit_{n\in\NN}\coev{\mathcal D}^R(\Phi_n)\longrightarrow \coev{\mathcal D}^R(\coev{\mathcal D}(\basering ))
  \end{equation}
  in $\modules$. 
\end{defn}
Note that $\contsgHH{\mathcal D}$ is well-defined up to a contractible space of choices, since the underlying Ind-object $\lbrace\Phi_n\rbrace_n\in\Ind{\mathcal D\otimes_\basering\mathcal D^\vee}$ is essentially unique according to~\cref{thm:internal-characterization-of-dualizability}.
\begin{rmk}
  Following Kontsevich, a dualizable stable $\basering$-linear category $\mathcal D$ is said to be \tdef{smooth} if $\coev{\mathcal D}$ is strongly continuous, or equivalently if $\coev{\mathcal D}(\basering )$ is compact. If $\mathcal D$ is smooth, then $\mathcal D$ is (a fortiori) second-countable and we find that
  \[
    \contsgHH{\mathcal D}\simeq 0,
  \]
  as one can see by noting that $\lbrace \Phi_n\rbrace_1^\infty$ can be taken to be the constant sequence
  \[
    \coev{\mathcal D}(\basering)\xtoo{\id}\coev{\mathcal D}(\basering)\xtoo{\id}\cdots.
  \]
  Tate--Hochschild cohomology can therefore be seen as measuring non-smoothness of dualizable categories.

  In topology, we have the terminologically awkward situation that smoothness (in Kontsevich's sense) is most closely related to the properness (i.e.~compactness) of topological spaces, in the sense that the failure of smoothness measured by the invariant $\contsgHH{\Shv X[\modules[\basering]]}$ has to do with the ends of $X$, see~\cref{thm:homology-at-infty} below. In particular, if $X$ is a compact and second-countable Hausdorff space, we will see that
  \[
    \contsgHH{\Shv X[\modules[\basering]]}\simeq 0.
  \]
  (This is not to say that $\Shv X[\modules[\basering]]$ is smooth on the nose if $X$ is compact; in fact $\Shv X[\modules[\basering]]$ is smooth if and only if $X$ is a finite discrete space, see~\cite{Harr2023}. Rather, it is the failure-of-smoothness as measured by invariants such as $\contsgHH{-}$ is related to compactness of $X$, since $X$ is compact if and only if the ind-system $\lbrace\Phi_n\rbrace$ is ind-constant.)
\end{rmk}

\subsection{Second-countability in topology}\label{subsec:second-countability}
We wish to show that the second-countability condition above precisely recovers the condition of the same name from point-set topology. The fact that the derived category of a second-countable locally compact Hausdorff space is second-countable in the sense of \cref{defn:second-countable} is straightforward using basic results about compactly exhaustible objects such as those in~\cite{Krause_Nikolaus_Putzstuck2024}, although it will be useful to give a detailed  proof. The other direction is more subtle, and will rely on the following celebrated theorem:
\begin{nonothm}[\v{S}ne\u{i}der~\cite{Sneider1945}]
  If $X$ is a compact Hausdorff space such that the diagonal $\Delta (X)\subseteq X\times X$ is $G_\delta$, then $X$ is metrizable.
\end{nonothm}
(Recall that a subset of a topological space is said to be \tdef{$G_\delta$} if it is a countable intersection of open subsets.) We refer to~\cite[Thm~2.13]{Gruenhage1984} for a comprehensive account of \v{S}ne\u{i}der's theorem and related results. Unfortunately, it seems that \v{S}ne\u{i}der's original article has not been digitized.

It follows immediately from \v{S}ne\u{i}der's theorem that if $X$ is $\sigma$-compact and locally compact Hausdorff such that $\Delta (X)\subseteq X\times X$ is $G_\delta$, then $X$ is metrizable, since a compact exhaustion $K_1\subseteq K_2\subseteq\cdots\subseteq X$ will have $K_n$ metrizable for each $n$. Alternatively, this can be deduced from Ceder's extension of \v{S}ne\u{i}der's theorem~\cite[Lem~5.4]{Ceder1961}, since a $\sigma$-compact locally compact Hausdorff space is paracompact and \v{C}ech complete~\cite{pi-base}.
We will also need the following lemma:
\begin{lem}\label{lem:aleph1-compact-implies-compactly-exhaustible}
  Let $X$ be a locally compact Hausdorff space. If $U\subseteq X$ is an open subset such that $h_U\in\Shv X[\modules[\basering]]$ is $\aleph_1$-compact, then $U$ is compactly exhaustible.
\end{lem}
\begin{proof}
  Since $X$ is locally compact Hausdorff, we have $h_U\simeq\colimit_{V\subseteq U}h_V$, where $V$ ranges over the poset of compactly exhaustible open subsets of $U$ ordered by inclusion. Since a countable union of compactly exhaustible open sets is again compactly exhaustible open, we find that this poset is $\aleph_1$-filtered. It follows that $h_U$ is a retract of $h_V$ for some compactly exhaustible $V\subseteq U$. Inspecting stalks, we find that we must then have $V = U$.
\end{proof}
\begin{thm}\label{thm:second-countability-recovers-notion-from-topology}
  A locally compact Hausdorff space $X$ is second-countable if and only if $\Shv X[\modules[\basering]]$ is second-countable.
\end{thm}
\begin{proof}
  We first show the ``if'' part, i.e.~if $X$ is second-countable in the non-commutative sense, then it is second-countable in the ordinary sense. The first step will be to show that $X$ is $\sigma$-compact.
As in~\cite{Ramzi2022}, we write
  \[
    \coev{}(\basering )\simeq\colimit_{i\in I} (f(i)\boxtimes g(i))\in\Shv X[\modules[\basering]]\otimes\Shv X[\modules[\basering]]
  \]
  where $f,g\colon I\to\Shv X[\modules[\basering]]$ are functors such that $f(i) = h_{U_i}\lbrack m_i\rbrack$ and $g(i) = h_{V_i}\lbrack n_i\rbrack$ for each $i\in I$, where $U_i,V_i\subseteq X$ are precompact open subsets and $m_i,n_i\in\ZZ$ are integers. We can decompose the colimit as
  \[
    \colimit_{i\in I} (f(i)\boxtimes g(i))\simeq\colimit_{J\subseteq I}\colimit_{j\in J} (f(j)\boxtimes g(j)),
  \]
  where the outer colimit on the right-hand side is taken over countable subsets $J\subseteq I$. The poset of countable subsets of $I$ is $\aleph_1$-filtered, and so by $\aleph_1$-compactness of $\coev{}(\basering )$ we find that $\coev{}(\basering )$ is a retract of $\colimit_{j\in J} (f(j)\boxtimes g(j))$ for some countable subset $J\subseteq I$. It follows in particular that $\Shv X[\modules[\basering]]$ is generated under colimits by $f(j)\simeq h_{U_j}\lbrack m_j\rbrack$ for $j\in J$. A support consideration now shows that $X = \bigcup_{j\in J}U_j = \bigcup_{j\in J}\overline{U_j}$, whence $X$ is $\sigma$-compact as claimed.

  Since $X$ is $\sigma$-compact it follows that also $X\times X$ is $\sigma$-compact. Consider the cofiber sequence
  \[
    h_{X\times X\setminus\Delta (X)}\too h_{X\times X}\too\Delta_*\basering_X
  \]
  in $\Shv{X\times X}$ By assumption $\Delta_*\basering_X$ is $\aleph_1$-compact. We have seen that $X\times X$ is $\sigma$-compact, whence $h_{X\times X}$ is $\aleph_1$-compact as well. By~\cite[Lem~2.3.14]{Krause_Nikolaus_Putzstuck2024}, we thus find that $h_{X\times X\setminus\Delta (X)}$ is $\aleph_1$-compact. \Cref{lem:aleph1-compact-implies-compactly-exhaustible} now implies that $X\times X\setminus\Delta (X)$ is compactly exhaustible. Pick a compact exhaustion
  \[
    K_1\subseteq K_2\subseteq\cdots \subseteq (X\times X)\setminus\Delta (X).
  \]
  We then have $\Delta (X) = \bigcap_{n=1}^\infty (X\times X)\setminus K_n$, whence $\Delta (X)$ is $G_\delta$. It follows from \v{S}ne\u{i}der's theorem that $X$ is metrizable, see the discussion preceding~\cref{lem:aleph1-compact-implies-compactly-exhaustible}. It follows that $X$ is second-countable as claimed.

  We now prove the other implication. Since $X$ is second-countable, so is $X\times X$ and the open subset $(X\times X)\setminus\Delta (X)$. But a second-countable locally compact Hausdorff space is $\sigma$-compact~\cite{pi-base}, so we may pick an exhaustion by opens
\[
  U_1\subseteq U_2\subseteq\cdots\subseteq (X\times X)\setminus\Delta (X)
\]
where $U_i$ is contained in a compact subset of $U_{i+1}$ for each $i\geq 0$. Let
\[
  V_1\subseteq V_2\subseteq\cdots\subseteq X
\]
be the same kind of exhaustion of $X$, and write
\[
  W_i = (V_i\times V_i)\cup U_{i+1}
\]
for $i\geq 0$.
\begin{figure}[h]
  \centering
  \def\svgwidth{0.8\textwidth}
  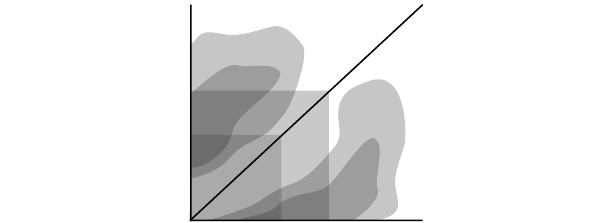
  \caption{The sequences of opens $\lbrace U_i\rbrace_i$ and $\lbrace V_i\rbrace_i$ depicted inside $X\times X$}
  \label{fig:ends-genus-two-three-cusps}
\end{figure}

The inclusions $U_i\subseteq W_i$ give rise to an $\NN\times\Delta^1$-shaped commutative diagram
  \[
    \begin{tikzcd}
      h_{U_1}\arrow[r]\arrow[d]
      & h_{U_2}\arrow[r]\arrow[d]
      & h_{U_3}\arrow[r]\arrow[d]
      & \cdots\\
      h_{W_1}\arrow[r]
      &h_{W_2}\arrow[r]
      &h_{W_3}\arrow[r]
      &\cdots
    \end{tikzcd}
  \]
  in which all arrows are compact. The colimits of the horizontal columns in this diagram are computed by (i) to be
  \begin{flalign*}
    \quad &h_{X\times X\setminus\Delta (X)}\simeq\colimit_n\left(h_{U_1}\too h_{U_2}\too h_{U_3}\too\cdots\right),\text{ and}&\\
    \quad &h_{X\times X}\simeq\colimit_n\left(h_{W_1}\too h_{W_2}\too h_{W_3}\too\cdots\right).&
  \end{flalign*}
  Letting $C_k = \cofib (h_{U_k}\to h_{W_k})$ for each $k\geq 1$, we thus get a sequential diagram
  \[
    C_1\too C_2\too C_3\too\cdots
  \]
  with
  \[
    \colimit_n(C_1\too C_2\too C_3\too\cdots)\simeq\cofib (h_{X\times X\setminus\Delta (X)}\too h_{X\times X})\simeq \Delta_*\basering_X,
  \]
  where the last equivalence comes from the usual recollement cofiber sequence. Furthermore, it follows from~\cite[Lem~2.3.7 \& 2.3.8]{Krause_Nikolaus_Putzstuck2024} that $C_k\to C_{k+2}$ is compact for each $k$. Since $\colimit_n(C_2\to C_4\to C_6\to\cdots)\simeq \Delta_*\basering_X$ by cofinality, we see in particular that $\Delta_*\basering_X$ is compactly-exhaustible and hence $\aleph_1$-compact as claimed.
\end{proof}

\subsection{Tate--Hochschild cohomology of topological spaces}\label{subsec:topology}
Our main objective in this subsection is to prove~\cref{mainthm:tate-hochschild-cohomology}. For the reader's convenience, we start by recalling the definition of singular homology at infinity, following~\cite[Ch~3]{Hughes_Ranicki1996}.

Let $X$ be a topological space, let $p$ be a nonnegative integer, and let $\operatorname{S}_p(X)$ denote the set of singular $p$-simplices in $X$. A \tdef{locally finite singular p-chain} is a formal $\ZZ$-linear sum \[f\, = \sum_{\sigma\in\operatorname{S}_p(X)} n_\sigma\sigma\] such that each $x\in X$ admits an open neighborhood $U$ for which the $U$-restricted support \[\supp_U(f) = \lbrace\sigma\mid n_\sigma\neq 0\text{ and }\sigma (|\Delta^p|)\cap U\neq\varnothing\rbrace\] is finite. The locally finite singular $p$-chains form an abelian group $\mathrm C_p^{\mathrm{lf}}(X;\ZZ )$, and by varying $p$ these assemble into a chain complex $\mdef{\mathrm C_\bullet^{\mathrm{lf}}(X;\ZZ )}$ with differential defined by the same formula as for ordinary singular chains. Note that the ordinary singular chain complex can be identified with the subcomplex of $\mathrm C_\bullet^{\mathrm{lf}}(X;\ZZ )$ comprised of formal sums $f$ with finite global support $\supp_X(f)$.
\begin{defn}
  The \tdef{singular chain complex at infinity} of $X$, denoted $\mdef{\mathrm C_\bullet^\infty (X;\ZZ )}$, is defined by
  \begin{equation}
  \label{eq:singular-to-locally-finite-chains}
  \mathrm C_\bullet^\infty (X;\ZZ ) = \fib\left(\mathrm C_\bullet(X;\ZZ )\too \mathrm C_\bullet^{\mathrm{lf}}(X;\ZZ )\right),
\end{equation}
where the fiber is computed in $\mathrm D(\ZZ )$. More generally, if $k$ is a commutative ring we put
\[
  \mdef{\mathrm C_\bullet^\infty (X;k)} = k\otimes_{\ZZ}\mathrm C_\bullet^\infty (X;\ZZ ).
\]
\end{defn}
The singular chain complex at infinity is covariantly functorial in proper maps. It is furthermore invariant under \tdef{proper homotopy equivalences}, meaning proper maps $f\colon X\to Y$ for which there exist proper maps
\[
  f'\colon Y\too X, \quad h\colon \lbrack 0,1\rbrack\times X\too X,\quad\text{and}\quad
  h'\colon \lbrack 0,1\rbrack\times Y\too Y
\]
with
\[
  f'f = h\vert_{0\times X},\quad ff' = h'\vert_{0\times Y},\quad\id = h\vert_{1\times X},\quad\text{and}\quad \id = h'\vert_{1\times Y}.
\]

For our purposes, it is more convenient to reformulate the definition of $\mathrm C_\bullet^\infty (X;\ZZ )$ in terms of sheaf theory. If $X$ is a triangularizable space, we can identify~\labelcref{eq:singular-to-locally-finite-chains} with the norm map
\begin{equation}
    \label{eq:borel-moore-to-homology-globsections}
    \cptGlob (X,\dualizingshf X)\too\Glob (X,\dualizingshf X),
  \end{equation}
  where $\omega_X\in\Shv X[\mathrm D(\ZZ )]$ is the Verdier dualizing sheaf. We can now state a reformulation of~\cref{mainthm:tate-hochschild-cohomology}:
\begin{thm}\label{thm:homology-at-infty}
  Let $k$ be a commutative ring. For a second countable locally compact triangularizable $d$-dimensional $k$-homology manifold $X$, there is an equivalence
  \begin{equation}
    \contsgHH{\Shv X[\mathrm D(k)]}\simeq\fib\left( \cptGlob (X,\dualizingshf X)\too\Glob (X,\dualizingshf X)\right)\lbrack -2d+1\rbrack
  \end{equation}
  in $\mathrm D(k)$.
\end{thm}
The theorem will require two lemmata:
\begin{lem}\label{lem:homology-manifold-exhaustion}
  Let $R$ be an $E_\infty$-ring spectrum and let $X$ be a $\sigma$-compact locally compact triangulated $d$-dimensional $R$-homology manifold. There exists an exhaustion of $X$ by compact subsets
  \[
    K_1\subseteq K_2\subseteq\cdots\subseteq X
  \]
  such that $X = \bigcup_1^\infty \interior K_j$ and the following conditions are satisfied for each $j\geq 1$:
  \begin{enumerate}[label=(\roman*)]
  \item $K_j\subseteq\interior K_{j+1}$;
  \item there is a subdivison of the triangulation of $X$ for which $K_j$ is a subcomplex; and 
  \item $K_j$ is a $d$-dimensional $R$-homology manifold with boundary \[\partial K_j = K_j\setminus\interior K_j.\]
  \end{enumerate}
\end{lem}
\begin{proof}
  Let $K\subseteq X$ be an arbitrary compact subcomplex. Then as in~\cite[Lem~2.60]{Luck_Schick2005} we can find $K'\subseteq X$ which is a subcomplex with respect to the barycentrically subdivised triangulation of $X$, such that $K\subseteq\interior K'$ and $K'$ is a compact $d$-dimensional $R$-homology manifold with $\partial K' = K'\setminus\interior K'$.\footnote{Although the statement in~\cite{Luck_Schick2005} is for $R = H\QQ$, the proof goes through without modifications for a general $E_\infty$-ring spectrum.} By repeatedly applying this to an exhaustion of $X$ by compact subcomplexes, we get the desired result.
\end{proof}
We will also need to know something about the behavior under $\Delta^!$ of sheaves of the form $h_{U\times U}\in\Shv{X\times X}[\mathrm D(k)]$, where $U\subseteq X$ is a suitable open $k$-homology submanifold. Let us first ponder the case $U = X$. Since $k$ is a commutative ring, there is a canonical equivalence
\begin{equation}
  \label{eq:square-of-dualizing-sheaf}
  \dualizingshf X\otimes\dualizingshf X\xtoo\sim k_X\lbrack -2d\rbrack
\end{equation}
natural in open $k$-homology submanifold inclusions, see~\cite[Prop~5.1.5(c)]{Schapira2026}. For more general reasons, there are also canonical natural equivalences
\[
  \Delta^*(\dualizingshf{X\times X})\simeq\Delta^*(\dualizingshf X\boxtimes\dualizingshf X )\simeq\dualizingshf X\otimes\dualizingshf X,
\]
see~\cite{Volpe2025}.
Applying the Verdier duality functor $D_X = \shHom{-}{\dualizingshf X}[X]$ to~\labelcref{eq:square-of-dualizing-sheaf} and invoking~\cite[Prop~3.3.7]{Dimca2004},
we find
\begin{equation}
  \label{eq:diagonal-corestriction}
  \Delta^!k_{X\times X}\xtoo\sim\dualizingshf X\lbrack-2d\rbrack,
\end{equation}
natural in open $k$-homology submanifold inclusions.
\begin{lem}\label{lem:diagonal-corestriction-calculation}
  Let $k$ be a commutative ring and let $X$ be a locally compact triangulated $d$-dimensional $k$-homology manifold. Let $K\subseteq X$ be a compact subcomplex that is itself a $d$-dimensional $k$-homology manifold with boundary $\partial K = K\setminus\interior K$. Put $U = \interior K$, let $\Delta\colon X\to X\times X$ denote the diagonal map, and let $j\colon U\hookrightarrow X$ be the inclusion. There is an equivalence
  \[
    j_!\dualizingshf U\lbrack -2d\rbrack\xtoo\sim \Delta^!h_{U\times U},
  \]
  natural in $U$.
\end{lem}
\begin{proof}
  We have the natural map
  \begin{equation}
    \label{eq:U-diagonal-corestriction}
    \Delta_U^!k_{U\times U}\xtoo\sim\dualizingshf U\lbrack -2d\rbrack
  \end{equation}
  as in~\labelcref{eq:diagonal-corestriction}, where $\Delta_U\colon U\to U\times U$ is the diagonal map. Applying $j_!$ to the inverse of~\labelcref{eq:U-diagonal-corestriction}, we produce
  \[
    j_!\dualizingshf U\lbrack -2d\rbrack\xtoo\sim j_!\Delta_U^!k_{U\times U}\simeq j_!j^!\Delta^!h_{U\times U}\too \Delta^!h_{U\times U},
  \]
  where the rightmost map is the counit of the $j_!\dashv j^!$ adjunction, and for the identification in the middle we have used that the unit map $k_{U\times U}\to (j\times j)^!(j\times j)_!k_{U\times U}\simeq (j\times j)^!h_{U\times U}$ is an equivalence, as well as the fact that $(j\times j)\circ \Delta_U = \Delta \circ j$.

  It suffices now to show that the counit map $j_!j^!\Delta^!h_{U\times U}\to\Delta^!h_{U\times U}$ is an equivalence, or in other words that
  \begin{equation}
    \label{eq:diagonal-corestriction-support}
    \supp (\Delta^!h_{U\times U})\subseteq U.
  \end{equation}
  Using the Verdier duality functor and the fact that the canonical map $h_{U\times U}\to D_{X\times X}D_{X\times X}h_{U\times U}$ is an equivalence, we rewrite
  \[
    \Delta^!h_{U\times U}\simeq D_X\Delta^*D_{X\times X}h_{U\times U},
  \]
  where we have again invoked~\cite[Prop~3.3.7]{Dimca2004}. Let $i\colon K\to X$ and $\Delta_K\colon K\to K\times K$ denote the inclusion and diagonal maps, respectively. Let also $j''\colon U\hookrightarrow K$ denote the inclusion, so that $j = i\circ j''$. We can then write $h_{U\times U}\simeq (i\times i)_!h_{U\times U}$, whence
  \[
    \Delta^*D_{X\times X}h_{U\times U}\simeq\Delta^*D_{X\times X}i_!h_{U\times U}\simeq i_!\Delta_K^*D_{K\times K}h_{U\times U}\simeq i_!\Delta_K^*(j''\times j'')_!\dualizingshf{U\times U},
  \]
  where for the last equivalence we use that $\dualizingshf{K\times K}\simeq (j''\times j'')_!\dualizingshf{U\times U}$ since $K\times K$ is a $k$-homology manifold with boundary $\partial (K\times K) = (K\times K)\setminus (U\times U)$. Base change gives
  \[
    i_!\Delta_K^*(j''\times j'')_!\dualizingshf{U\times U}\simeq j_!\Delta_U^*\dualizingshf{U\times U}\simeq j_!k_U\lbrack -2d\rbrack,
  \]
  where the last equivalence is~\labelcref{eq:square-of-dualizing-sheaf}. But $x^*j_!\simeq 0$ for $x\not\in U$, giving~\labelcref{eq:diagonal-corestriction-support}.
\end{proof}
We are now ready for the main event of this section:
\begin{proof}[Proof of~\cref{thm:homology-at-infty}]
  Let $p\colon X\to\pt$ and $\Delta\colon X\to X\times X$ denote the point projection and diagonal maps, respectively.
  Let $K_1\subseteq K_2\subseteq\cdots\subseteq X$ be an exhaustion by $k$-homology manifolds as in~\cref{lem:homology-manifold-exhaustion}, and let $V_i = \interior K_i$ denote their interiors. Let $\lbrace U_i\rbrace_1^\infty$ be an exhaustion of $(X\times X)\setminus\Delta (X)$ as in the proof of~\cref{thm:second-countability-recovers-notion-from-topology}, and let $W_i = (V_i\times V_i)\cup U_{i+1}$. We saw in the aforementioned proof that
  \[
    C_2\too C_4\too C_6\too\cdots
  \]
  is a compact exhaustion of $\Delta_*k_X$, where
  \[
    C_i = \cofib\left( h_{U_i}\too h_{W_i}\right).
  \]
  Note that $\coev{}^R(h_{U_i}) \simeq p_*\Delta^!h_{U_i}\simeq 0$ since $U_i\cap\Delta (X) = \varnothing$, and for the same reason the image under $\coev{}^R$ of the counit $h_{V_i\times V_i}\to h_{W_i}$ is an equivalence. Applying $\coev{}^R$ to $C_i\to \Delta_*k_X$ therefore spits out the canonical map
  \[
    p_*\Delta^!h_{V_i\times V_i}\too p_*\Delta^!h_{X\times X}.
  \]
  From~\cref{lem:diagonal-corestriction-calculation}, we get a natural identification of this map with the canonical map
  \[
    p_*(j_i)_!\dualizingshf{V_i}\lbrack -2d\rbrack\too p_*\dualizingshf{X}\lbrack -2d\rbrack,
  \]
  Furthermore, since $V_i$ is precompact in $X$, the norm map
  \[
    p_!(j_i)_!\dualizingshf{V_i}\lbrack -2d\rbrack\too p_*(j_i)_!\dualizingshf{V_i}\lbrack -2d\rbrack
  \]
  is an equivalence. Since $p_!$ is a left adjoint and hence preserves colimits, we thus get an identification of
  \[
    \colimit \coev{}^RC_{2i}\too \coev{}^R\coev{}(k)
  \]
  with
  \[
    \cptGlob (X,\dualizingshf X)\lbrack -2d\rbrack\too\Glob (X,\dualizingshf X)\lbrack -2d\rbrack
  \]
  as desired.
\end{proof}
The homology at infinity of $X$ is closely related to the ends of $X$. To make this precise, we will need an additional assumption. Recall that Quinn calls a locally compact Hausdorff space $X$ \tdef{forward tame} if there is a closed cocompact subspace $Y\subseteq X$ and a proper map $\varphi\colon Y\times\lbrack 0,\infty )\to X$ such that $\varphi\vert_{Y\times 0}$ agrees with the inclusion of $Y$ into $X$, see~\cites[][]{Quinn1988}[][Ch.~12]{Hughes_Ranicki1996}. For instance, this holds if $X$ is the interior of a closed manifold with boundary.
\begin{cor}\label{cor:number-of-ends}
  Let $X$ be a locally compact triangularizable $d$-dimensional homology manifold. If $X$ is forward tame, then
  \[
    \operatorname{rank}\left(\pi_{-2d}\inftyHH{\Shv X[\mathrm D(\ZZ )]}\right) = \#\lbrace\text{ends of }X \rbrace
  \]  
\end{cor}
\begin{proof}
  This follows from~\cref{thm:homology-at-infty} and~\cite[p.~137f]{Hughes_Ranicki1996}
\end{proof}

\subsection{Tate--Hochschild cohomology of algebras}\label{subsec:algebra}
Throughout this subsection, $\basering$ is commutative discrete ring.
\subsubsection{Review of Tate--Hochschild cohomology of discrete algebras}
Recall that the bounded derived category of a discrete associative $\basering$-algebra $A$ is the Dwyer--Kan localization
\[
  \modules[A]^{\mathrm{bd}} = \operatorname{Ch}^{\mathrm{bd}}_A\lbrack \text{quasi-isomorphisms}^{-1}\rbrack,
\]
where $\operatorname{Ch}^{\mathrm{bd}}_A$ is the 1-category of bounded chain complexes of finitely-generated $A$-modules. 
\begin{defn}[{Buchweitz~\cite{Buchweitz1986}}]\label{defn:singular-category}
  The \tdef{singular category}\footnote{Buchweitz uses the term ``stabilized category.''} of $A$ is the Verdier quotient
\[
  \mdef{\operatorname{Sg}(A)} = \modules[A]^{\mathrm{bd}}/\operatorname{Perf}(A),
\]
where $\operatorname{Perf}(A) = \modules[A]^{\aleph_0}$ is the category of perfect complexes.
\end{defn}
\begin{rmk}
  Note that the singular category of a Noetherian commutative ring $A$ vanishes iff $A$ has finite global dimension, or in other words iff $A$ is regular~\cite[Lem~10.110.8]{Stacks2026}, whence the name ``singular category.''
\end{rmk}
\begin{rmk}
  For $M$ and $N\in\modules[A]^{\mathrm{bd}}$, the Ext-group
\begin{equation}
  \label{eq:Tate-cohomology-groups}
  \operatorname{Ext}^i_{\operatorname{Sg}(A)}(M,N) = \pi_{-i}\Hom{M}{N}[\operatorname{Sg}(A)]
\end{equation}
is referred to as the $i$-th \tdef{Tate cohomology group} of $M$ with values in $N$.
In particular, if $A = \basering\lbrack G\rbrack$ for some finitely-generated group $G$, then $\operatorname{Sg}(A)$ is equal to the stable module category of $G$, so~\labelcref{eq:Tate-cohomology-groups} recovers the usual notion of Tate cohomology.
\end{rmk}
Suppose now that the enveloping algebra $A^e = A\otimes_\basering A^{\mathrm{op}}$ is left Noetherian. Inspired by the definition of Hochschild cohomology 
\[
  \operatorname{HH}(A,M) = \Hom{A}{M}[\modules[A^e]]\in\modules[\basering],
\]
$M\in\modules[A^e]$, one can then give the following ``singular'' or ``Tate'' variant of~\cref{defn:singular-category}:
\begin{defn}[{Eu--Schedler~\cite{Eu_Schedler2009}, Bergh--Jorgensen~\cite{Bergh_Jorgensen2013}}, Keller~\cite{Keller2018}]
  Given $M\in\modules[A^e]^{\mathrm{bd}}$, the \tdef{Tate--Hochschild cohomology} of $A$ with coefficients in $M$ is the endomorphism object
\[
  \mdef{\sgHH{A}[M]} = \Hom{A}{M}[\operatorname{Sg}(A^e)]\in\modules[\basering].
\]
\end{defn}
From this definition, it is not clear that $\sgHH{-}$ should admit an interesting extension to a large class of dualizable stable categories. In particular, for something like the derived category of a locally compact Hausdorff space, there is no obvious replacement for the singular category. 
\begin{rmk}
  Wang~\cite{Wang2021} defined the Tate--Hochschild cohomology of an arbitrary (non-Noetherian) associative algebra, see also~\cite{Keller2018}. Our definition of Tate--Hochschild cohomology does not make sense in this generality.
\end{rmk}
\subsubsection{Comparison with our non-commutative definition}
The derived category $\modules[A]$ of an associative $\basering$-algebra $A$ is compactly-generated and thus a dualizable stable $\basering$-linear category, with $(\modules[A])^\vee\simeq\modules[A^{\mathrm{op}}]$. There is a canonical equivalence
\begin{equation}
  \label{eq:bimodules-is-modules-over-enveloping}
  \modules[A]\otimes_{\basering}\modules[A^{\mathrm{op}}]\simeq\modules[A^e].
\end{equation}
Under this equivalence, the object $\coev{}(\basering)\in\modules[A]\otimes_{\basering}\modules[A^{\mathrm{op}}]$ identifies with $A$ equipped with the usual $A^e$-module structure.
It follows (see~\cref{rmk:compactness-in-cpt-gen-categories}) that
\begin{prop}\label{prop:second-countable-algebra}
  Let $A$ be a discrete associative $\basering$-algebra. $\modules[A]$ is second-countable if and only if $A$ admits an $A^e$-linear resolution
  \begin{equation}
    \label{eq:fpinfty-resolution}
    \cdots\too P_2\too P_1\too P_0\too A,
  \end{equation}
  where each $P_i$ is a finitely-generated projective $A^e$-module. 
\end{prop}
\begin{exam}
  If $A = \basering\lbrack G\rbrack$ for some group $G$, then an extension-of-scalars argument shows that the existence of a resolution~\labelcref{eq:fpinfty-resolution} is equivalent to the condition that $G$ is $\mathrm{FP}_\infty$.
\end{exam}
\begin{exam}
  If $A^e$ is Noetherian, then $A$ automatically admits a resolution of the form~\labelcref{eq:fpinfty-resolution}. Indeed, $A$ is a cyclic $A^e$-module generated by the unit $1\in A$. 
\end{exam}
\begin{prop}
  Let $A$ be a discrete associative $\basering$-algebra satisfying the equivalent conditions of~\cref{prop:second-countable-algebra}. There is an equivalence
  \[
    \contsgHH{\modules[A]}\simeq\sgHH{A}[A].
  \]
\end{prop}
\begin{proof}
  Let $P_\bullet\to A$ be a resolution by finitely-generated projective $A^e$-modules.
  Let
  \[
    \operatorname{Fil}^0P_\bullet\longrightarrow\operatorname{Fil}^{-1}P_\bullet\longrightarrow \operatorname{Fil}^{-2}P_\bullet\longrightarrow\cdots\longrightarrow A
  \]
  denote the associated (``stupid'') filtration of $A$ in $\modules[A^e]$, so that $\operatorname{Fil}^nP_\bullet$ is the underlying object in $\modules[A^e]$ of the chain complex with
  \[
    (\operatorname{Fil}^{n}P_\bullet )_i =
    \begin{cases}
      P_i,&\text{if }i\geq n,\\
      0,&\text{else}
    \end{cases}
  \]
  and the obvious differential. This is a perfect complex, i.e.~a compact object of $\modules[A^e]$, whence by definition
  \[
    \contsgHH{\modules[A]}\simeq\colimit_n\Hom{A}{\cofib (\operatorname{Fil}^{-n}P_\bullet\to A)}[A^e].
  \]
  The stated equivalence is then~\cite[Lem~2.4]{Keller2018}.
\end{proof}


\section{The Gorenstein locus}\label{sec:gorenstein-locus}
A celebrated theorem of Bondal and Orlov states that a smooth projective variety $X$ can be reconstructed from its derived category $\mathrm{DQCoh}(X)$ if either the canonical bundle $K_X$ or its dual $K_X^\vee$ is ample~\cite{Bondal_Orlov2001}. The main insight in their proof is that $\mathrm{DQCoh}(X)$ has an intrinsic duality, implemented by the Serre endofunctor
\[
  K_X\otimes {-}\colon \mathrm{DQCoh}(X)\longrightarrow \mathrm{DQCoh}(X).
\]
The existence of an essentially unique Serre endofunctor is a general property of dualizable stable $\infty$-categories, and it is this that allows us to prove~\cref{mainthm:gorenstein-locus} and its corollary~\cref{mainthm:reconstruction}(i) by exploiting that the Serre endofunctor of $\Shv X[\mathrm D(\ZZ )]$ for a locally compact Hausdorff space $X$ is given by
\[
  \dualizingshf X\otimes{-}\colon \Shv X[\mathrm D(\ZZ )]\longrightarrow \Shv X[\mathrm D(\ZZ )],
\]
where $\dualizingshf X$ is the Verdier dualizing sheaf of $X$.
\subsection{The homology manifold locus of a space}
Let $\basering$ be a connective $E_\infty$-ring such that $\pi_0\basering$ is a principal ideal domain. (The condition on $\pi_0\basering$ ensures that an $\basering$-module is invertible iff it is of the form $\basering\lbrack n\rbrack$ for some $n\in\ZZ$.)
\begin{defn}\label{defn:homological-regular-locus}
  Let $X$ be a locally compact Hausdorff space. The \tdef{$\basering$-homology manifold locus} of $X$, denoted $\mdef{\hreg X[\basering ]}\subseteq X$, is the (open) subset of points $x\in X$ such that the restriction $\dualizingshf X\vert_U$ is invertible for some open neighborhood $U$ of $x$.
\end{defn}
\begin{rmk}
  Note that by~\cite[Cor~2.5.4.12]{Martini_Wolf2022}, a point $x\in X$ belongs to the $\basering$-homology manifold locus of $X$ iff $\dualizingshf X[x]$ is invertible and $\dualizingshf X$ is locally constant in a neighborhood of $x$.
\end{rmk}
The homology manifold locus is especially well-behaved for topological spaces that admit triangulations.
\begin{prop}
  Let $X$ be a locally compact triangularizable space. Then $\hreg X[\basering]$ is nonempty and its complement $X\setminus\hreg X[\basering]$ is a subcomplex of $X$ for any choice of triangulation. 
\end{prop}
\begin{proof}
  For non-emptiness, note that $\hreg X[\basering]$ contains the interior of any maximal simplex $\Delta\subseteq X$, since $\interior\Delta\cong\RR^d$ for some $d\geq 0$, whence $\dualizingshf X\vert_{\interior\Delta}\simeq\dualizingshf{\interior\Delta}\simeq\basering\lbrack d\rbrack$~\cite[Prop~6.21]{Volpe2025}.

  For the statement about $X\setminus\hreg X$, fix an arbitrary triangulation of $X$. Note that $\dualizingshf X[\basering]$ is constant on the relative interior of any simplex. Hence $\hreg X[\basering]$ and $X\setminus\hreg X[\basering]$ are both unions of relative interiors of simplices of $X$. Since $X\setminus\hreg X[\basering]$ is closed in $X$, it must be a subcomplex.
\end{proof}
\begin{rmk}
  For a pseudomanifold $X$, the homology manifold locus $\hreg X[\basering]$ is equal to the top stratum of the canonical filtrations introduced by Goresky and MacPherson in their proof of the topological invariance of intersection cohomology~\cite{Goresky_Macpherson1983}, see also~\cite{Rourke_Sanderson1999}. (Unlike the lower strata, the top stratum of these filtrations does not depend on the choice of perversity function.) Indeed, the top stratum of these filtrations is defined to be the set of points $x\in X$ such that $\dualizingshf X$ is locally constant in a neighborhood of $x$. This is equivalent to the condition that we impose, since any $x\in X$ is contained in a maximal simplex $\Delta$, and on the interior $\interior\Delta$ we have $\dualizingshf X\vert_{\interior\Delta}\simeq\basering\lbrack\dim\Delta\rbrack$ by~\cite[Prop~6.21]{Volpe2025}; thus if $\dualizingshf X$ is locally constant near $x$, we must also have $\dualizingshf X[x]\simeq\basering\lbrack\dim\Delta\rbrack$.
\end{rmk}
\begin{exam}\label{exam:hreg-not-stalkwise-condition}
  Note that $\hreg X$ is (typically) smaller than the set of points $x\in X$ having $\dualizingshf X[x]$ invertible. Indeed, the latter is generally neither an $\basering$-homology manifold nor an open subset of $X$ (although it is a locally closed subset since $\dualizingshf X$ is constant on relative interiors of simplices).

  As an example, let $S^1\vee I$ denote the lollipop graph formed by attaching an interval $I = \lbrack 0,1\rbrack$ along one of its endpoints to the circle. Consider the cone $X = C(S^1\vee I)$, and put
  $X ' = \lbrace x\in X\mid\dualizingshf X[x]\text{ is invertible}\rbrace$.
  \begin{figure}[h]
    \centering
    \def\svgwidth{0.8\textwidth}
\begingroup%
  \makeatletter%
  \providecommand\color[2][]{%
    \errmessage{(Inkscape) Color is used for the text in Inkscape, but the package 'color.sty' is not loaded}%
    \renewcommand\color[2][]{}%
  }%
  \providecommand\transparent[1]{%
    \errmessage{(Inkscape) Transparency is used (non-zero) for the text in Inkscape, but the package 'transparent.sty' is not loaded}%
    \renewcommand\transparent[1]{}%
  }%
  \providecommand\rotatebox[2]{#2}%
  \newcommand*\fsize{\dimexpr\f@size pt\relax}%
  \newcommand*\lineheight[1]{\fontsize{\fsize}{#1\fsize}\selectfont}%
  \ifx\svgwidth\undefined%
    \setlength{\unitlength}{283.46456693bp}%
    \ifx\svgscale\undefined%
      \relax%
    \else%
      \setlength{\unitlength}{\unitlength * \real{\svgscale}}%
    \fi%
  \else%
    \setlength{\unitlength}{\svgwidth}%
  \fi%
  \global\let\svgwidth\undefined%
  \global\let\svgscale\undefined%
  \makeatother%
  \begin{picture}(1,0.186)%
    \lineheight{1}%
    \setlength\tabcolsep{0pt}%
    \put(0,0){\includegraphics[width=\unitlength,page=1]{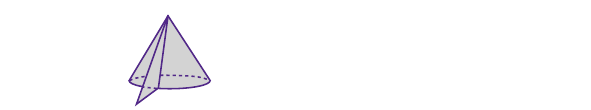}}%
    \put(0.27167336,0.16933502){\color[rgb]{0,0,0}\makebox(0,0)[lt]{\lineheight{1.25}\smash{\begin{tabular}[t]{l}$\infty$\end{tabular}}}}%
    \put(0.27215951,0.01129738){\color[rgb]{0,0,0}\makebox(0,0)[lt]{\lineheight{1.25}\smash{\begin{tabular}[t]{l}$\pt$\end{tabular}}}}%
    \put(0.70033341,0.01130119){\color[rgb]{0,0,0}\makebox(0,0)[lt]{\lineheight{1.25}\smash{\begin{tabular}[t]{l}$\pt$\end{tabular}}}}%
    \put(0,0){\includegraphics[width=\unitlength,page=2]{cone_regular_locus.pdf}}%
    \put(0.70005792,0.16933746){\color[rgb]{0,0,0}\makebox(0,0)[lt]{\lineheight{1.25}\smash{\begin{tabular}[t]{l}$\infty$\end{tabular}}}}%
    \put(0,0){\includegraphics[width=\unitlength,page=3]{cone_regular_locus.pdf}}%
  \end{picture}%
\endgroup%

    \caption{The complements of $X'$ (left) and $\hreg X$ (right) as subsets (highlighted in purple) of the 1-skeleton of $X$, with $0$-cells depicted as fattened points}
    \label{fig:cone-regular-locus}
  \end{figure}
  For $\basering\neq 0$, the $\basering$-homology manifold locus $\hreg X[\basering]$ is equal to the complement of the cone base $S^1\vee I$ and the 1-cells $C(\partial I)$, see~\cref{fig:cone-regular-locus}. In particular, $\hreg X[\basering]$ does not contain the cone point $\infty$. On the other hand, $\dualizingshf X[\infty]\simeq\widetilde{\mathrm C}_*(\Sigma (S^1\vee I);\basering )\simeq\basering\lbrack 2\rbrack$, and indeed $X' = \hreg X[\basering ]\cup\lbrace\infty\rbrace$.

  $X$ is also an example of a space for which the recursively defined filtration
  \begin{equation}
    \label{eq:recursively-defined-filtration}
    X\supset X\setminus\hreg X[\basering]\supset (X\setminus\hreg X[\basering])\setminus \hreg{(X\setminus\hreg X[\basering])}[\basering]\supset\cdots
  \end{equation}
  is not a $\mathrm H$-filtration in the sense of \cite{Rourke_Sanderson1999}; that is, it fails to verify the ``equisingularity'' condition, which asks that the restriction of the dualizing sheaf on the $n$-skeleton to the $i$-stratum is locally constant for $i\leq n$. Indeed, letting $x_0\in S^1\vee I$ denote the attaching point, we find that~\labelcref{eq:recursively-defined-filtration} is given by
  \[
    X\supset (S^1\vee I)\cup C(\partial I)\supset\lbrace x_0\rbrace.
  \]
  In particular, the cone point $\infty$ belongs to the 1-stratum $Y_1 = ((S^1\vee I)\cup C(\partial I))\setminus\lbrace x_0\rbrace$. But $\dualizingshf X\vert_{Y_1}$ is not locally constant near $\infty$. On $C(x_0)\setminus\lbrace\infty\rbrace$, the sheaf $\dualizingshf X$ restricts to the constant sheaf with value $\basering\lbrack 2\rbrack\oplus \basering\lbrack 2\rbrack$, whereas on the other leg $C(\partial I)\setminus C(x_0)$ it restricts to the zero sheaf, and at $\infty$ we have $\dualizingshf X[\infty]\simeq\basering \lbrack 2\rbrack$ as we saw above. 
\end{exam}
The following description will be useful later:
\begin{lem}\label{lem:characterization-of-homological-regular-locus}
  Let $X$ be a locally compact triangularizable space. A point $x\in X$ belongs to $\hreg X$ if and only if there is some $d\geq 0$ such that for every simplex $\Delta\subseteq X$ containing $x$, there is an equivalence $\dualizingshf X\vert_{\relint\Delta}\simeq \basering\lbrack d\rbrack_{\relint\Delta}$. 
\end{lem}
\begin{proof}
  Let $X '\subseteq X$ denote the subset consisting of points satisfying the conditions given in the lemma. Clearly $\hreg X\subseteq X'$.
  For the reverse inclusion, note that the conditions guarantee that $x\in X'$ has an open neighborhood $U\ni x'$ such that $\dualizingshf X[y]\simeq \basering\lbrack d\rbrack$ for all $y\in U$. As in~\cites[][Prop~7.9]{Scholze2026}[][Thm~4.15]{Land_Volpe2026}, we find that $\dualizingshf X$ is locally constant on $U$, completing the proof.
\end{proof}
\subsection{Recollection on Serre functors}\label{subsec:serre-functors}
We wish to define the Gorenstein locus of a dualizable category as the subcategory consisting of objects that behave as if they are ``supported on homologically smooth points.'' In~\cref{subsec:gorenstein-locus} we will make this idea precise using the notion of a Serre functor, briefly reviewed here.
\begin{defn}
  Let $\basering$ be an $E_\infty$-ring, and let $\mathcal D$ be a dualizable $\basering$-linear stable category. The \tdef{Serre functor} of $\mathcal D$, denoted
  \[
    \mdef{S_{\mathcal D}}\colon\mathcal D\longrightarrow\mathcal D,
  \]
  is defined as the composite
  \begin{equation}
    \label{eq:serre-functor-formula}
    \mathcal D\simeq\mathcal D\otimes_{\basering}\modules[\basering]\xrightarrow{\id\otimes\ev{}^R}
    \mathcal D\otimes_{\basering} \mathcal D^\vee\otimes_\basering\mathcal D\simeq\mathcal D^\vee\otimes_\basering\mathcal D\otimes_\basering\mathcal D\xrightarrow{\ev{}\otimes\id}\modules[\basering]\otimes\mathcal D\simeq\mathcal D,
  \end{equation}
  where $\ev{}\colon\mathcal D^\vee\otimes_\basering\mathcal D\to\modules[\basering]$ denotes the (essentially unique) evaluation pairing for $\mathcal D$, and $\ev{}^R$ denotes its $\basering$-linear right adjoint.
\end{defn}
Alternatively, we can write the formula for the Serre functor using the following string diagram:
\begin{equation*}
  \begin{tikzpicture}[
    baseline = (current bounding box.center)
    line width=10pt,
    box/.style={
      draw,
      rectangle,
      minimum width=0.8cm,
      minimum height=0.8cm,
      inner sep=0pt
    }
    ]

    \node[box, thick] (evR) at (2,0.5) {$\ev{}^R$};
    \node[box, thick] (ev)  at (7,2) {$\ev{}$};

    \path (evR.east) ++(0, 0.22) coordinate (evRoutTop);
    \path (evR.east) ++(0,-0.22) coordinate (evRoutBot);

    \path (ev.west) ++(0, 0.22) coordinate (evInTop);
    \path (ev.west) ++(0,-0.22) coordinate (evInBot);

    \draw[thick] (-0.5,2.55) -- (4.1,2.55) node [very near start,yshift=8pt]{$\mathcal D$}
    .. controls (4.9,2.55) and (5,1.7) ..
    (evInBot) ;

    \draw [thick] (evRoutTop) to[out=20,in=180] (evInTop);

    \draw [thick] (evRoutBot) -- (9.5,0.22) node [very near end,xshift=4pt,yshift=8pt]{$\mathcal D$};

    \node at (3.45,1.5) {$\mathcal D^\vee$};
  \end{tikzpicture}
\end{equation*}
\smallskip

Serre functors were originally defined by Bondal and Kapranov in the context of small triangulated categories~\cite{Bondal_Kapranov1989}. The key ingredient in Bondal and Orlov's reconstruction theorem for smooth (anti-)Fano varieties~\cite{Bondal_Orlov2001} is the observation that the Serre functor on the derived category of coherent sheaves is given by tensoring with the dualizing sheaf. A similar description holds for the derived category of sheaves on a locally compact Hausdorff space:
\begin{prop}
  Let $X$ be a locally compact Hausdorff space and let $\basering$ be an $E_\infty$-ring. The Serre functor on $\Shv X[\modules[\basering]]$ is given by
  \[
    \dualizingshf X\otimes ({-})\colon\Shv X[\modules[R]]\too \Shv X[\modules[R]].
  \]
\end{prop}
\begin{proof}
  By covariant Verdier duality, the evaluation pairing for $\Shv X[\modules[\basering]]$ can be described as
  \[
    \Shv X[\modules[\basering]]\otimes_\basering\Shv X[\modules[\basering]]\xrightarrow\sim
    \Shv{X\times X}[\modules[\basering]]\xrightarrow{\Delta^*}
    \Shv X[\modules[\basering]]\xrightarrow{p_!}
    \modules[\basering],
  \]
  where $\Delta\colon X\to X\times X$ is the diagonal embedding and $p\colon X\to\pt$ is the projection to a point. We thus find that the right adjoint $\ev{}^R$ is given by the composition
  \[
    \modules[\basering]\xrightarrow{p^!}
    \Shv X[\modules[\basering]]\xrightarrow{\Delta_*}
    \Shv{X\times X}[\modules[\basering]]\xrightarrow\sim
    \Shv X[\modules[\basering]]\otimes_\basering\Shv X[\modules[\basering]].
  \]
  Unfolding definitions, we then find that the Serre functor is given by
  \begin{equation}
    \label{eq:serre-functor-on-sheaves-intermediate-step}
    (p_2)_!(\Delta_*\dualizingshf X\otimes p_1^*({-})),
  \end{equation}
  where $p_1$ and $p_2\colon X\times X\to X$ denote the projection onto the first and second factor, respectively. Applying the projection formula to~\labelcref{eq:serre-functor-on-sheaves-intermediate-step} and using that $p_1\Delta = p_2\Delta = \id$ now produces the desired formula.
\end{proof}
\subsection{The Gorenstein locus of a dualizable category}\label{subsec:gorenstein-locus}
Let $\basefield$ be a field.
\begin{defn}\label{defn:regular-locus}
  Let $\mathcal D$ be a dualizable $\basefield$-linear stable $\infty$-category, and let $S_{\mathcal D}\curvearrowright\mathcal D$ denote its Serre functor. An object $x\in\mathcal D$ is \tdef{regular} if the map
  \[
    \Hom{x}{y}[\mathcal D]\xtoo{S_{\mathcal D}}\Hom{Sx}{Sy}[\mathcal D]
  \]
  is an equivalence for all $y\in\mathcal D$. We refer to the full subcategory of $\mathcal D$ spanned by regular objects as the \tdef{Gorenstein locus} of $\mathcal D$, and denote it by $\mdef{\reg{\mathcal D}}\subseteq\mathcal D$.
\end{defn}
Intuitively, the Gorenstein locus $\reg{\mathcal D}$ consists of objects $F$ such that the Serre functor $S_{\mathcal D}$ is invertible ``near the support'' of $F$. Our goal in this section is to turn this intuition into a precise description of the Gorenstein locus of the derived category of sheaves on a space $X$, assuming that $X$ admits the structure of a locally finite simplicial complex; we will see that an object of this derived category is regular iff its support is contained in the $\basefield$-homology manifold locus. We do not know whether this description remains true if one drops the assumption that $X$ admits a triangulation, but it is always true that a sheaf supported on the homology manifold locus is regular, as the first proposition of this subsection shows. For the remainder of this subsection, we abbreviate $\hreg X = \hreg X[\basefield]$.
\begin{prop}\label{prop:supported-on-gorenstein-implies-regular}
  Let $X$ be a locally compact Hausdorff space, and let $j\colon\hreg X\hookrightarrow X$ denote the inclusion of its $\basefield$-homology manifold locus. If $F\in\Shv{\hreg X}[\mathrm D(\basefield )]$, then $j_!F$ belongs to the Gorenstein locus of $\Shv{X}[\mathrm D(\basefield )]$.
\end{prop}
\begin{proof}
  Let $G\in\Shv X$ be arbitrary. Using the projection formula and the $j_!\dashv j^!\simeq j^*$ adjunction, we get a commutative diagram
  \[
    \begin{tikzcd}
      \Hom{j_!F}{G}[\Shv X]\arrow[r,"\dualizingshf X\otimes {-}"]\arrow[d,"\simeq"]
      & \Hom{\dualizingshf X\otimes j_!F}{\dualizingshf X \otimes G}[\Shv X]\arrow[d,"\simeq"] \\
      \Hom{F}{j^*G}[\Shv{\hreg X}]\arrow[r,"j^*\dualizingshf X\otimes {-}"]
      & \Hom{j^*\dualizingshf X\otimes F}{j^*\dualizingshf X\otimes j^*G}.
    \end{tikzcd}
  \]
  By construction the sheaf $j^*\dualizingshf X$ is locally constant with invertible stalks and hence an invertible object of $\Shv{\hreg X}[\mathrm D(\basefield )]$, whence the bottom horizontal arrow is an equivalence. It follows that the top horizontal arrow is an equivalence too.
\end{proof}
In the opposite direction, we have the following partial result:
\begin{prop}\label{prop:invertible-stalks}
  Let $X$ be a locally compact Hausdorff space. If $F\in\Shv X[\mathrm D(\basefield )]$ is regular, then $\dualizingshf X[x]$ is invertible for each $x\in\supp F$.
\end{prop}
The proof will require a lemma:
\begin{lem}\label{lem:hom-equivalence-implies-invertible}
  Let $V,W\in\mathrm D(\basefield )$ with $V\neq 0$. If
  \[
    {-}\otimes W\colon\Hom{V}{\basefield}[\mathrm D(\basefield )]\too\Hom{V\otimes W}{W}[\mathrm D(\basefield )]
  \]
  is an equivalence, then $W$ is invertible. 
\end{lem}
\begin{proof}
  The collection
  \[
    \lbrace V'\mid {-}\otimes W\colon\Hom{V'}{\basefield}[\mathrm D(\basefield )]\longrightarrow\Hom{V'\otimes W}{W}[\mathrm D(\basefield )]\text{ is an equivalence}\rbrace
  \]
  is closed under retracts. Since $V\neq 0$ and $\basefield$ is a field, there is some $m\in\ZZ$ such that $\basefield\lbrack m\rbrack$ is a retract of $V$. We thus have that
  \[
    {-}\otimes W\colon\Hom{\basefield\lbrack m\rbrack}{\basefield}[\mathrm D(\basefield )]\longrightarrow\Hom{\basefield\lbrack m\rbrack\otimes W}{W}[\mathrm D(\basefield )]
  \]
  is an equivalence, so by shifting down we find that the unit map
  \begin{equation}
    \label{eq:unit-map-for-endo-algebra}
    \basefield\longrightarrow\Hom{W}{W}[\mathrm D(\basefield )]
  \end{equation}
  is an equivalence of algebras. In particular, we must have $W\neq 0$. But then there is $n\in\ZZ$ such that $\basefield\lbrack n\rbrack$ is a retract of $W$. Since~\labelcref{eq:unit-map-for-endo-algebra} is an equivalence and $\basefield$ has no non-trivial idempotents, we conclude that $W\simeq\basefield\lbrack n\rbrack$.
\end{proof}
\begin{proof}[Proof of~\cref{prop:invertible-stalks}]
  Consider the skyscraper sheaf $x_*\basefield$. We have a commutative diagram
  \begin{equation*}
    \begin{tikzcd}
      \Hom{F}{x_*\basefield }[\Shv X[\mathrm D(\basefield )]]\arrow[d,"\simeq"]\arrow[r,"S"]
      & \Hom{S(F)}{S(x_*\basefield)}[\Shv X[\mathrm D(\basefield )]]\arrow[d,"\simeq"]\\
      \Hom{F_x }{\basefield}[\mathrm D(\basefield )]\arrow{r}{\dualizingshf{X}[x] \otimes -}
      & \Hom{\dualizingshf{X}[x] \otimes F_x }{\dualizingshf{X}[x] }[\mathrm D(\basefield )].
    \end{tikzcd}
  \end{equation*}
  The claim now follows from~\cref{lem:hom-equivalence-implies-invertible} applied to the bottom horizontal map.
\end{proof}
\begin{thm}\label{thm:noncomm-regularity}
  Let $X$ be a locally compact triangularizable space, and let \[j\colon\hreg X\hookrightarrow X\] denote the inclusion of the homology manifold locus.

  The Gorenstein locus $\reg{\Shv X[\mathrm D(\basefield )]}$ is equal as a full subcategory of $\Shv X[\mathrm D(\basefield )]$ to the essential image of $j_!$.
\end{thm}
\begin{proof}
  We have already seen that the image of $j_!$ is contained in $\reg{\Shv X}$, see~\cref{prop:supported-on-gorenstein-implies-regular}. Suppose that $F\in\Shv X$ is regular. We must show that $F$ lies in the essential image of $j_!$, or equivalently that $\supp F$ is contained in the homology manifold locus $\hreg X$.

  Fix $x\in\supp F$. By~\cref{prop:invertible-stalks}, we have
  \[
    \dualizingshf X[x]\simeq\basefield\lbrack d\rbrack
  \]
  for some $d\geq 0$. By~\cref{lem:characterization-of-homological-regular-locus}, it will suffice to show that $\dualizingshf X\vert_{\relint\Delta}\simeq\basefield\lbrack d\rbrack_{\relint\Delta}$ for every simplex $\Delta$ containing $x$. We will split the proof into two cases:
  
\smallskip
\noindent\textit{Case 1.} $\supp F\cap\relint\Delta = \varnothing$.
\smallskip

  Note that we must have $x\not\in\relint\Delta$, and in particular this forces $\dim\Delta > 0$. We can therefore pick an embedded arc $a\colon\lbrack 0,1\rbrack\to X$ such that $a(0) = x$ and $a(t)\in\relint\Delta$ for $t>0$. Let $j'\colon (0,1)\hookrightarrow \lbrack 0,1\rbrack$ denote the inclusion. We consider the sheaf
  \[
    J = (j')_!\basefield_{(0,1)}\in\Shv{\lbrack 0,1\rbrack}.
  \]
  Note also that $J\simeq \dualizingshf{\lbrack 0,1\rbrack}\lbrack -1\rbrack$. We let $e\colon\lbrace 0\rbrace\hookrightarrow\lbrack 0,1\rbrack$ denote the inclusion of the endpoint, so that in particular $x = a\circ e$. There is a map
   \begin{equation}
    \label{eq:comparing-costalks}
    \varphi\colon e^!J\otimes \dualizingshf X[x]\to e^!(J\otimes a^* \dualizingshf X)
  \end{equation}
  adjoint to
  \[
    e_!(e^!J\otimes \dualizingshf X[x]) \simeq e_!(e^!J\otimes e^*a^*\dualizingshf X)
    \simeq e_!e^!J\otimes a^*\dualizingshf X\to J\otimes a^*\dualizingshf X,
  \]
  where the second equivalence is the projection formula and the rightmost map is the counit for the $v_!\dashv v^!$ adjunction. Chasing through adjunctions, we find that there is a commutative diagram
  \[
    \begin{tikzcd}
      \Hom{F}{a_!J}[\Shv X]\arrow[r,"\dualizingshf X\otimes {-}"]\arrow[d,"\simeq"]
      & \Hom{\dualizingshf X\otimes F}{\dualizingshf X\otimes a_!J}[\Shv X]\arrow[d,"\simeq"]\\
      \Hom{F_x}{e^!J}[\mathrm D(\basefield )]\arrow[r]
      & \Hom{F_x}{e^!(J\otimes a^*\dualizingshf X)\otimes\dualizingshf X[x]^{-1}}[\mathrm D(\basefield )],
    \end{tikzcd}
  \]
  where the bottom horizontal arrow is post-composition with $\varphi\otimes\dualizingshf X[x]^{-1}$. Here we have used that $F\vert_{\relint\Delta}\simeq 0$ by assumption to identify $a^*F\simeq e_*F_x$.
    \begin{figure}[h]
  \centering
  \def\svgwidth{0.8\textwidth}
\begingroup%
  \makeatletter%
  \providecommand\color[2][]{%
    \errmessage{(Inkscape) Color is used for the text in Inkscape, but the package 'color.sty' is not loaded}%
    \renewcommand\color[2][]{}%
  }%
  \providecommand\transparent[1]{%
    \errmessage{(Inkscape) Transparency is used (non-zero) for the text in Inkscape, but the package 'transparent.sty' is not loaded}%
    \renewcommand\transparent[1]{}%
  }%
  \providecommand\rotatebox[2]{#2}%
  \newcommand*\fsize{\dimexpr\f@size pt\relax}%
  \newcommand*\lineheight[1]{\fontsize{\fsize}{#1\fsize}\selectfont}%
  \ifx\svgwidth\undefined%
    \setlength{\unitlength}{283.46456693bp}%
    \ifx\svgscale\undefined%
      \relax%
    \else%
      \setlength{\unitlength}{\unitlength * \real{\svgscale}}%
    \fi%
  \else%
    \setlength{\unitlength}{\svgwidth}%
  \fi%
  \global\let\svgwidth\undefined%
  \global\let\svgscale\undefined%
  \makeatother%
  \begin{picture}(1,0.27893999)%
    \lineheight{1}%
    \setlength\tabcolsep{0pt}%
    \put(0,0){\includegraphics[width=\unitlength,page=1]{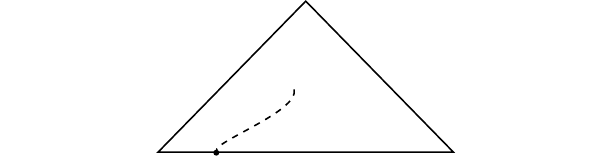}}%
    \put(0.36237895,-0.00767762){\color[rgb]{0,0,0}\makebox(0,0)[lt]{\lineheight{1.25}\smash{\begin{tabular}[t]{l}$x$\end{tabular}}}}%
    \put(0.58313716,0.03770202){\color[rgb]{0,0,0}\makebox(0,0)[lt]{\lineheight{1.25}\smash{\begin{tabular}[t]{l}$\relint\Delta$\end{tabular}}}}%
    \put(0.48286525,0.07403595){\color[rgb]{0,0,0}\makebox(0,0)[lt]{\lineheight{1.25}\smash{\begin{tabular}[t]{l}$e$\end{tabular}}}}%
  \end{picture}%
\endgroup%

  \caption{The arc segment $e$}
  \label{fig:arc-costalk}
\end{figure}
Regularity of $F$ implies that the top horizontal arrow is an equivalence, and since $F_x\not\simeq 0$, we conclude that $\varphi$ is also an equivalence. But $e^!J\otimes\dualizingshf X[x]\simeq \basefield\lbrack d-1\rbrack$ and $e^!(J\otimes a^*\dualizingshf X)\simeq v^!(J\otimes M)\simeq \dualizingshf X[\Delta]\lbrack -1\rbrack$, where we denote $\dualizingshf X[\Delta] = \Glob (\relint\Delta ,\dualizingshf X)$ and use that $\dualizingshf X\vert_{\relint\Delta}$ is locally constant, and so by contractibility of $\relint\Delta$ it is in fact locally constant with value $\dualizingshf X[\Delta]$. We thus conclude that $\dualizingshf X\vert_{\relint\Delta}\simeq\basefield\lbrack d\rbrack$ as desired.

\smallskip
\noindent\textit{Case 2.} $\supp F\cap\relint\Delta \neq \varnothing$.
\smallskip

By~\cref{prop:invertible-stalks}, we find that $\dualizingshf X\vert_{\relint\Delta}\simeq\basefield\lbrack d'\rbrack$ for some $d'\geq 0$. As in Case 1, we pick an embedded arc $a\colon\lbrack 0,1\rbrack\to X$ such that $a(0) = x$ and $a(t)\in\relint\Delta$ for $t>0$, and we let $e\colon\lbrace 0\rbrace\hookrightarrow\lbrack 0,1\rbrack$ denote the endpoint inclusion. We also let $j'\colon (0,1\rbrack\hookrightarrow\lbrack 0,1\rbrack$ denote the inclusion of the open complement of $\lbrace 0\rbrace$. We have a recollement fiber sequence
  \begin{equation}
    \label{eq:recollement-along-arc}
    (j')_!(j')^*a^*\dualizingshf X\to a^*\dualizingshf X\to e_*\dualizingshf X[x].
  \end{equation}
  We wish to show that $d' = d$, so assume for contradiction that this is not the case. The fiber sequence~\labelcref{eq:recollement-along-arc} has attaching map
  \[
    e_*\basefield\lbrack d\rbrack \simeq e_*\dualizingshf X\to (j')_!(j')^*a^*\dualizingshf X\lbrack 1\rbrack\simeq (j')_!\basefield\lbrack d'+1\rbrack,
  \]
  which is adjoint to a map
  \[
    \basefield\lbrack d\rbrack\to e^!(j')_!\basefield\lbrack d'+1\rbrack\simeq \basefield\lbrack d'\rbrack
  \]
  in $\mathrm D(\basefield )$, and since $\basefield$ is a field and $d\neq d'$, this map is zero. It follows that~\labelcref{eq:recollement-along-arc} splits, giving
  \[
    a^*\dualizingshf X\simeq e_*\basefield\lbrack d\rbrack\oplus (j')_!\basefield\lbrack d'\rbrack.
  \]
  Write $F_a = a^*F$ for short. We claim that the counit map $\varepsilon\colon (j')_!(j')^*F_a\to F_a$ also splits. To see this, consider the commutative diagram
  \[
    \begin{tikzcd}
      \Hom{(j')_!(j')^*F_a}{(j')_!(j')^*F_a}
      \arrow[r,"a^*\dualizingshf X\otimes {-}"]
      & \Hom{(j')_!(j')^*F_a\lbrack d'\rbrack}{(j')_!(j')^*F_a\lbrack d'\rbrack}\\
      \Hom{F_a}{(j')_!(j')^*F_a}\arrow[r,"a^*\dualizingshf X\otimes {-}"]\arrow[d]\arrow[u]
      & \Hom{e_*F_x\lbrack d\rbrack\oplus (j')_!(j')^*F_a\lbrack d'\rbrack}{(j')_!(j')^*F_a\lbrack d'\rbrack}\arrow[d]\arrow[u]\\
      \Hom{F}{a_*(j')_!(j')^*F_a}\arrow[r,"\dualizingshf X\otimes {-}"]
      & \Hom{\dualizingshf X\otimes F}{\dualizingshf X\otimes a_*(j')_!(j')^*F_a},
    \end{tikzcd}
  \]
  where the top vertical arrows are given by precomposition with $\varepsilon$.
  Here the bottom horizontal arrow is an equivalence by regularity of $F$ and the bottom vertical arrows are equivalences by adjunction and the projection formula, so it follows that the middle horizontal arrow is an equivalence. In particular there is a map $\eta\colon F_a\to (j')_!(j')_!F_a$ lifting the projection $e_*F_x\lbrack d\rbrack\oplus (j')_!(j')^*F_a\lbrack d'\rbrack\to (j')_!(j')^*F_a\lbrack d'\rbrack$. Chasing through the top square we find that this map must split $\varepsilon$ as claimed.

  We next use regularity of $F$ against the sheaf $a_*\basefield_{\lbrack 0,1\rbrack}$. Write
  \begin{align*}
    &\dualizingshf X[a] = a^*\dualizingshf X, \quad\dualizingshf X[a]^0 = e_*\dualizingshf X[x],\quad \dualizingshf X[a]^{>0} =  j'_!(j')^*\dualizingshf X[a], \\
    &F_a^0 = e_*F_x,\quad\text{and}\quad F_a^{>0} = j'_!(j')^* F_a.
  \end{align*}
  We have seen that
  \[
    \dualizingshf X[a]\simeq\dualizingshf X[a]^0\oplus \dualizingshf X[a]^{>0}\quad\text{and}\quad
    F_a\simeq F_a^0\oplus F_a^{>0}.
  \]
  Using the projection formula, we calculate
  \[
    \dualizingshf X[a]^0\otimes F_a^0\simeq F_a^0\lbrack d\rbrack,\quad
    \dualizingshf X[a]^{>0}\otimes F_a^{>0}\simeq F_a^{>0}\lbrack d'\rbrack,\quad
    \dualizingshf X[a]^{>0}\otimes F_a^0\simeq 0,\,\,\text{and}\quad
    \dualizingshf X[a]^{0}\otimes F_a^{>0}\simeq 0.
  \]  
  From this we see that the regularity equivalence 
  \[
    \Hom{F}{a_*\basefield_{\lbrack 0,1\rbrack}}\xrightarrow{\dualizingshf X\otimes{-}}\Hom{\dualizingshf X\otimes F}{\dualizingshf X\otimes a_*\basefield_{\lbrack 0,1\rbrack}}
  \]
  is the direct sum of maps
  \begin{equation}
    \label{eq:partthree-stalk-map}
    \Hom {F_a^0}{\basefield_{\lbrack 0,1\rbrack}}\to \Hom {F_a^0\lbrack d\rbrack}{\dualizingshf X[a]^0}\oplus \Hom {F_a^0\lbrack d\rbrack}{\dualizingshf X[a]^{>0}}
  \end{equation}
  and
  \[
    \Hom {F_a^{>0}}{\basefield_{\lbrack 0,1\rbrack}}\to \Hom {F_a^{>0}\lbrack d\rbrack}{\dualizingshf X[a]^0}\oplus \Hom {F_a^{>0}\lbrack d'\rbrack}{\dualizingshf X[a]^{>0}}
  \]
  both of which must therefore be equivalences. A local calculation shows that $e^!\basefield_{\lbrack 0,1\rbrack}\simeq 0$, so
  \[
    \Hom {F_a^0}{\basefield_{\lbrack 0,1\rbrack}} \simeq\Hom{e_*F_x}{\basefield_{\lbrack 0,1\rbrack}}\simeq\Hom{F_x}{e^!\basefield_{\lbrack 0,1\rbrack}}\simeq 0,
  \]
  see e.g.~\cite[Lem~2.26, p.~91]{Harr2025c}. On the other hand,
  \[
     \Hom {F_a^0\lbrack d\rbrack}{\dualizingshf X[a]^0}\simeq\Hom{F_x}{\basefield}.
   \]
   Since~\labelcref{eq:partthree-stalk-map} is an equivalence, we must then have
   $\Hom{F_x}{\basefield}\simeq 0$, whence $F_x\simeq 0$, contradiction.
\end{proof}

\section{Reconstruction of graphs and surfaces}
By applying the results of the previous sections, we can now prove~\cref{mainthm:reconstruction}. If $(X,Y)$ is an FM pair, we will say that $Y$ is an \tdef{FM partner} to $X$.

We first make some prelimary observations.
\begin{defn}
  Let $\mathcal C$ be a stable category. A full subcategory $\mathcal C_0\subseteq\mathcal C$ is called \tdef{clopen} if there is a decomposition $\id_{\mathcal C}\simeq F_0\oplus F_1$ such that $\mathcal C_0$ is equivalent to the essential image of $F_0$ (as a subobject of $\mathcal C$). The category $\mathcal C$ is said to be \tdef{connected} if every clopen subcategory of $\mathcal C$ is zero or equivalent to $\mathcal C$ (as a subobject of $\mathcal C$).
\end{defn}
Let $k$ be a commutative ring without non-trivial idempotents. Using the equivalence
\[
  \begin{tikzcd}
    \Shv{X\times X}[\mathrm D(k)]\arrow[r]
    &  \Fun^L(\Shv X[\mathrm D(k)],\Shv X[\mathrm D(k)] )\\[-15pt]
    K\arrow[r,mapsto]
    & (p_2)_!(K\otimes p_1^*{-}),
  \end{tikzcd}
\]
where $p_1$ and $p_2\colon X\times X\to X$ denote the projections onto the first and second factor, it is straightforward to show that a subcategory of $\Shv X[\mathrm D(k)]$ is clopen if and only if it is the subcategory of sheaves supported on a clopen subset of $X$.
\begin{defn}
  Let $\mathcal D$ be a dualizable stable category with Serre functor $S_{\mathcal D}$. The \tdef{Calabi--Yau dimension} of $\mathcal D$, denoted $\mdef{\dim_{\mathrm{CY}}(\mathcal D)}$, is given by
  \[
    \dim_{\mathrm{CY}}(\mathcal D)
    = \lbrace d\in\ZZ_{\geq 0}\colon S_{\mathcal D}({-})\vert_{\mathcal C}\simeq ({-})\lbrack d\rbrack\text{ for some nonzero clopen }\mathcal C\subseteq\mathcal D\rbrace\subseteq\ZZ_{\geq 0}.
  \]
\end{defn}
Using that every $\finitefield_2$-homology manifold has a $\finitefield_2$-orientation, we see that
\begin{prop}\label{prop:dim-of-spaces}
  Let $X$ be a locally compact triangularizable space. Then $X$ is finite dimensional if and only if $\dim_{\mathrm{CY}}(\reg{\Shv X[\mathrm D(\finitefield_2)]})$ is nonempty and bounded, and in the affirmative case
  \[
    \dim X = \max \dim_{\mathrm{CY}}(\reg{\Shv X[\mathrm D(\finitefield_2)]})
  \]
\end{prop}
\begin{proof}[Proof of~\cref{mainthm:reconstruction}]
  (i) Let $\Gamma$ be a second-countable 1-dimensional triangularizable space, and let $\Gamma '$ be a triangularizable FM partner to $\Gamma$. From~\cref{thm:second-countability-recovers-notion-from-topology,prop:dim-of-spaces} we see that $\Gamma '$ is itself second-countable and 1-dimensional. Note that every point in $\Gamma$ (and $\Gamma '$) has an open neighborhood homeomorphic to a space of the form
  \[
    \mdef{C_v} =
    \begin{cases}
      \bigcup_{i=1}^v\RR_{\geq 0}e_i\subseteq\RR^v,&\text{if }v\geq 1,\\
      \pt,&\text{if }v=0,
    \end{cases}
  \]
  where $e_1,\dots ,e_v$ is the canonical basis of $\RR^v$. We refer to $C_v$ as the \tdef{corolla of valence $v$}. We say that the \tdef{valence} of a point $x$ in $\Gamma$ (or $\Gamma '$) is $v$ if $x$ has an open neighborhood homeomorphic to $C_v$. Every point has a well-defined (unique) valence. Furthermore, we have
  \[
    \hreg\Gamma = \lbrace x\in\Gamma\mid x \text{ has valence two or zero}\rbrace
  \]
  and $\Gamma\setminus\hreg\Gamma$ is a discrete closed subspace of $\Gamma$. By the classification of second-countable 1-dimensional manifolds, every component of $\hreg\Gamma$ is either homeomorphic to $\pt$, $S^1$ or $\RR$. Labelling each component of $\hreg\Gamma$ by its homeomorphism type defines a function
  \[
    L\colon\pi_0\hreg\Gamma\too\lbrace \pt , S^1,\RR \rbrace
  \]
  Furthermore, if $E\in L^{-1}(\RR )$, either end of $E$ may converge to a point in $\Gamma\setminus\hreg\Gamma$ (an \tdef{endpoint}). Note that both ends may converge to the same point. That is, the endpoint data consists of between zero and two points with multiplicity in $\Gamma\setminus\hreg\Gamma$, with total multiplicity $\leq 2$. This is recorded by a function
  \begin{equation}
    \label{eq:endpoint-data-function}
    L^{-1}(\RR )\too
    ((\Gamma\setminus\hreg\Gamma )_+\times(\Gamma\setminus\hreg\Gamma )_+ )/\mathfrak C_2,
  \end{equation}
  where $\mathfrak C_2$ is the cyclic group of order two, which acts on the product by swapping factors. The homeomorphism type of $\Gamma$ is uniquely determined by (i) the sets $\pi_0\hreg\Gamma$ and $\Gamma\setminus\hreg\Gamma$, the labelling function $L$, and the endpoint data~\labelcref{eq:endpoint-data-function}. 

  From~\cref{mainthm:gorenstein-locus} and the observatiosn preceding this proof, we see that
  \[
    \begin{tikzcd}
      \pi_0\hreg\Gamma\arrow[r]
      &\bigg\lbrace
    \begin{array}{l}
      \text{minimal non-trivial clopen}\\
      \text{subcategories of }\reg{\Shv\Gamma[\mathrm D(k)]}
    \end{array}
    \bigg\rbrace \\[-15pt]
    E\arrow[r,mapsto]
    & \text{subcategory of sheaves supported on }E
    \end{tikzcd}
  \]
  is a bijection for an arbitrary field $k$. Let us write $\mathcal D_E\subseteq\reg{\Shv\Gamma[\mathrm D(k)]}\subseteq\Shv\Gamma[\mathrm D(k)]$ for the subcategory of sheaves supported on $E$. We can tell the homeomorphism type of $E$ from Efimov's theorem~\cite[Thm~6.1]{Efimov2025}, which gives
  \[
    \mathrm{HH}(\mathcal D_E)
    \simeq\cptGlob (X,k_X)\simeq
    \begin{cases}
      k\oplus k\lbrack -1\rbrack,&\text{if }E\cong S^1,\\
      k\lbrack -1\rbrack,&\text{if }E\cong\RR,\\
      k,&\text{if } E\cong\pt.
    \end{cases}
  \]
  It remains only to determine the endpoint data~\labelcref{eq:endpoint-data-function}. The Verdier sequence
  \[
    \reg{\Shv\Gamma[\mathrm D(k)]}\too\Shv\Gamma[\mathrm D(k)]\too \Shv\Gamma[\mathrm D(k)] / \reg{\Shv\Gamma[\mathrm D(k)]}
  \]
  is equivalent to the recollement
  \[
    \Shv{\hreg\Gamma}[\mathrm D(k)]\xtoo{j_!}
    \Shv{\Gamma}[\mathrm D(k)]\xtoo{i^*}
    \Shv{\Gamma\setminus\hreg\Gamma}[\mathrm D(k)],
  \]
  where $j\colon\hreg\Gamma\hookrightarrow\Gamma$ and $i\colon \Gamma\setminus\hreg\Gamma\to\Gamma$ denote the inclusions. In particular, we recover the gluing functor
  \[
    i^*j_*\colon\Shv{\hreg\Gamma}[\mathrm D(k)]\too\Shv{\Gamma\setminus\hreg\Gamma}[\mathrm D(k)].
  \]
  Let $\mathcal D_E\subseteq\reg{\Shv\Gamma[\mathrm D(k)]}$ be a minimal non-trivial clopen subcategory corresponding to a component $E\subseteq\hreg\Gamma$, and let $\mathcal D_x\subseteq \Shv\Gamma[\mathrm D(k)] / \reg{\Shv\Gamma[\mathrm D(k)]}$ be a minimal non-trivial clopen subcategory corresponding to a component/point $x\in\Gamma\setminus\hreg\Gamma$. Consider the composition
  \begin{equation}
    \label{eq:restricted-gluing}
    \mathcal D_E\hookrightarrow\reg{\Shv\Gamma[\mathrm D(k)]}\xtoo{i^*j_*}\Shv{\Gamma\setminus\hreg\Gamma}[\mathrm D(k)]\to\mathcal D_x,
  \end{equation}
  where the rightmost functor is left (and in fact also right) adjoint to the inclusion. A straightforward sheaf calculation shows that~\labelcref{eq:restricted-gluing} is nonzero if and only if $x\in\overline E\setminus E$. As for the multiplicity,~\labelcref{eq:restricted-gluing} is itself the gluing functor for the recollement associated with the open-closed decomposition of $E\cup\lbrace x\rbrace$ into $E$ and $\lbrace x\rbrace$. Hence we recover $\Shv{E\cup\lbrace x\rbrace}[\mathrm D(k)]$ from~\labelcref{eq:restricted-gluing} as in e.g.~\cite[Rem~A.8.12]{Lurie2017}. But then we can again use $\contsgHH{-}$ or the existence of nonzero compact objects to test whether $E\cup\lbrace x\rbrace$ is homeomorphic to $\RR_{\geq 0}$ or $S^1$, or in other words whether the multiplicity of $x$ is one or two (respectively) as an endpoint of $E$.

  (ii) Let $\Sigma$ be a second-countable surface, and let $X$ be a triangularizable FM partner to $\Sigma$. Using the discussion preceding the proof, we may assume that both are connected. We will first assume that $\Sigma$ is without boundary (but possibly non-compact). Then the Serre functor $\dualizingshf X\otimes{-}$ must act invertibly on $\reg{\Shv X[\mathrm D(\ZZ )]}$, whence by~\cite[Cor~2.5.4.12]{Martini_Wolf2022} $X$ is a homology manifold. We also know that it is 2-dimensional from~\cref{prop:dim-of-spaces}. But 2-dimensional homology manifolds are topological manifolds by a theorem of Wilder, and so by the classification of second-countable 2-dimensional topological manifolds it suffices to observe that we can recover the compactly supported cohomology $\mathrm C^\bullet_{\mathrm c}(\Sigma ;\ZZ )$ as the $\ZZ$-linear Hochschild homology complex $\mathrm{HH}(\Shv{\Sigma}[\mathrm D(\ZZ )])$ by Efimov's theorem~\cite[Thm~6.1]{Efimov2025} and the number of punctures of $\Sigma$ as the rank of $\pi_{-3}\contsgHH{\Shv \Sigma [\mathrm D(\ZZ )]}$ by~\cref{cor:number-of-ends}. If $\Sigma$ has boundary, note that the Verdier sequence
  \[
    \hreg{\Shv{\Sigma}[\mathrm D(k )]}\too \Shv{\Sigma}[\mathrm D(k )]\too
    \Shv{\Sigma}[\mathrm D(k )] / \hreg{\Shv{\Sigma}[\mathrm D(k )]}
  \]
  (for an arbitrary field $k$) is equivalent to the recollement
  \[
    \Shv{\Sigma\setminus\partial\Sigma}[\mathrm D(k )]\too
    \Shv{\Sigma}[\mathrm D(k )]\too
    \Shv{\partial\Sigma}[\mathrm D(k )].
  \]
  We have seen how to recover the homeomorphism types of both $\partial\Sigma$ (which is 1-dimensional, and more specifically a disjoint union of circles) and $\Sigma\setminus\partial\Sigma$. But then $\Sigma$ is the unique (up to homeomorphism) surface with $\mathrm C^\bullet_{\mathrm c}(\Sigma ;\ZZ )\simeq \mathrm{HH}(\Shv{\Sigma}[\mathrm D(\ZZ )])$,
  \[
    b = \#\lbrace
    \text{minimial non-trivial clopen subcategories of }\Shv{\partial\Sigma}[\mathrm D(k )]
    \rbrace
  \]
  many boundary components, and
  \[
    \operatorname{rank}\pi_{-3}\contsgHH{\Shv{\Sigma\setminus\partial\Sigma} [\mathrm D(\ZZ )]} - b
  \]
  many punctures.
  
  (iii) Let $X$ be a triangularizable FM partner to $S^3$. As in (ii), we see that $X$ must be a connected closed 3-dimensional homology manifold. But triangularizable 3-dimensional homology manifolds are topological manifolds. By Perelman's proof of the 3-dimensional Poincar\'e conjecture~\cite{Perelman2002,Perelman2003}, it will therefore suffice to show that $X$ is simply-connected. By Perelman's proof of the geometrization conjecture combined with Hempel's theorem on fundamental groups of Haken manifolds~\cite[Thm~1.1, see Cor~1.2]{Hempel1987}, we know that the fundamental group $\pi_1X$ is residually finite. If for contradiction $\pi_1X\neq 0$, then $\pi_1X$ must admit a non-trivial finite dimensional complex representation $V$. But then $\lbrack V\rbrack$ and the class of the trivial representation $\lbrack\CC^{\mathrm{triv}}\rbrack$ are $\ZZ$-linearly independent in $\mathrm K_0(\mathrm D^{\mathrm b}(\CC\lbrack\pi_1X\rbrack ))$, whence
  \begin{equation}
    \label{eq:non-trivial-class-in-k-thy}
    \operatorname{rank}\mathrm K_0(\mathrm D^{\mathrm b}(\CC\lbrack\pi_1X\rbrack ))\geq 2.
  \end{equation}

  On the other hand, from the classification of compact objects in derived categories of locally compact spaces~\cite[Thm~2.6]{Harr2025} and Lurie's monodromy theorem~\cite[Thm~A.1.15]{Lurie2017}, we see that
  \[
    \Fun (\Pi_\infty S^3,\mathrm D^{\mathrm b}(\CC ))\simeq
    \Shv{S^3}[\mathrm D(\CC )]^{\aleph_0}\simeq
    \Shv X[\mathrm D(\CC )]^{\aleph_0}\simeq \Fun (\Pi_\infty X,\mathrm D^{\mathrm b}(\CC )),
  \]
  where $\Pi_\infty({-})$ denotes the fundamental $\infty$-groupoid functor. From Barwick's theorem of the heart~\cite{Barwick2015} applied to the pointwise t-structure on local systems, we also have
  \[
    \mathrm K_0(\Fun (\Pi_\infty X,\mathrm D^{\mathrm b}(\CC )))\simeq
    \mathrm K_0(\mathrm D^{\mathrm b}(\CC\lbrack\pi_1X\rbrack ))
  \]
  and
  \[
    \mathrm K_0(\Fun (\Pi_\infty S^3,\mathrm D^{\mathrm b}(\CC )))\simeq
    \mathrm K_0(\mathrm D^{\mathrm b}(\CC\lbrack\pi_1S^3\rbrack ))\simeq\ZZ.
  \]
  But this contradicts~\labelcref{eq:non-trivial-class-in-k-thy}, finishing the proof.
\end{proof}
We mentioned in the introduction that there are many non-trivial FM pairs among non-triangularizable spaces. Aside from non-sober spaces, non-trivial FM pairs are also abundant among Alexandrov spaces: 
\begin{exam}\label{exam:verdier-duality-for-posets}
Let $P$ be the face poset of a locally finite simplicial complex (or a Verdier poset in the sense of~\cite{Aoki2023}) with $P\not\cong P^{\mathrm{op}}$,\footnote{e.g.~a simplex of dimension $\geq 2$} and let $\operatorname{Alex}(P)$ and $\operatorname{Alex}(P^{\mathrm{op}})$ denote the associated Alexandrov/finitely-generated spaces. Then
  \[
    \Shv{\operatorname{Alex}(P)}[\modules[\basefield]]\simeq\Fun (P,\modules[\basefield] )\simeq \Fun (P^{\mathrm{op}},\modules[\basefield])\simeq\Shv{\operatorname{Alex}(P^{\mathrm{op}})}[\modules[\basefield]]
  \]
  by~\cites[][Thm~7.7]{Curry2018}[][Thm~A]{Aoki2023} but $\operatorname{Alex}(P)\not\cong\operatorname{Alex}(P^{\mathrm{op}})$ since the functor $\operatorname{Alex}({-})$ is fully faithful.
\end{exam}

\printbibliography{}
\end{document}